\documentclass[a4paper,reqno, 11pt]{amsart}  
\usepackage[DIV=12, oneside]{typearea}
\usepackage[utf8]{inputenc}
\usepackage[T1]{fontenc}
\usepackage[english]{babel}
\usepackage[centertags]{amsmath}
\usepackage{amstext,amssymb,amsopn,amsthm}
\usepackage{nicefrac, esint}
\usepackage{mathrsfs}
\usepackage{dsfont}
\usepackage{bbm}
\usepackage{thmtools}
\usepackage{graphicx}

\usepackage[backgroundcolor=white, bordercolor=blue,
linecolor=blue]{todonotes}

\usepackage[colorlinks=true, linkcolor=blue, citecolor=black]{hyperref}
\usepackage{pgfplots}
\usetikzlibrary{patterns}
\usetikzlibrary{arrows}
\usepackage{enumitem}

\addto\extrasenglish{}
\addto\extrasenglish{}

\declaretheorem[name=Theorem, numberwithin=section]{theorem}
\newtheorem{corollary}[theorem]{Corollary}
\newtheorem{lemma}[theorem]{Lemma}
\newtheorem{proposition}[theorem]{Proposition}
\newtheorem{definition}[theorem]{Definition}

\newtheorem{example}{Example}
\newtheorem*{example*}{Example}
\newtheorem*{robremark*}{Robustness remark}

\declaretheoremstyle[bodyfont=\normalfont]{remark-style}
\declaretheorem[name={Remark}, style=remark-style, unnumbered]{remark}

\numberwithin{equation}{section}

\theoremstyle{plain}

\newtheorem{fact}[theorem]{Fact}

\newcommand{\N}{\mathds{N}}
\newcommand{\R}{\mathds{R}}

\newcommand{\Z}{\mathds{Z}}

\newcommand{\ind}{\mathds{1}}

\newcommand{\loB}{\nu_*}
\newcommand{\upB}{\nu^*}

\newcommand{\sE}{\mathscr{E}}

\newcommand{\cB}{\mathcal{B}}
\newcommand{\cW}{\mathcal{W}}
\newcommand{\cH}{\mathcal{H}}
\newcommand{\cA}{\mathcal{A}}

\newcommand{\cS}{\mathcal{S}}
\newcommand{\cE}{\mathcal{E}}

\newcommand{\eps}{\varepsilon}
\newcommand{\supl}{\sup\limits}

\newcommand{\suml}{\sum\limits}

\newcommand{\il}{\int\limits}

\newcommand{\iil}{\iint\limits}

\newcommand{\wt}{\widetilde}
\newcommand{\ov}{\overline}
\newcommand{\Rd}{{\mathds{R}^d}}

\newcommand{\as}[2][r_0]{(K#2,$#1$)}

\newcommand{\epsx}{\eps_x}
\newcommand{\epsy}{\eps_y}

\newcommand{\opn}[1]{\operatorname{#1}}

\def\ilRd{\int\limits_{\R^d}}

\DeclareMathOperator{\lin}{lin}
\DeclareMathOperator{\dist}{dist}
\DeclareMathOperator{\diam}{diam}
\DeclareMathOperator{\supp}{supp}

\DeclareMathOperator{\osc}{osc}
\DeclareMathOperator{\radius}{radius}

\renewcommand{\d}{\textnormal{d}}

\newcommand{\1}{{\mathbbm{1}}}

\begin{document}

\title[Elliptic nonlocal operators] {Regularity estimates for elliptic nonlocal operators}

\author{Bart{\l}omiej Dyda}
\author{Moritz Kassmann}

\address{Department of Pure and Applied Mathematics, Wroc{\l}aw University 
of Technology, Wybrze\.ze Wyspia\'nskiego 27, 50-370 Wroc{\l}aw, Poland}
\address{Fakult\"{a}t f\"{u}r Mathematik\\Universit\"{a}t Bielefeld\\Postfach 
100131\\D-33501 Bielefeld}


\thanks{Both authors have been supported by the German Science Foundation DFG through \\ SFB 701. The first author was additionally supported by NCN grant 2012/07/B/ST1/03356.}



\email{bdyda@pwr.edu.pl}
\email{moritz.kassmann@uni-bielefeld.de}

\date{\today}

\begin{abstract}
We study weak solutions to nonlocal equations governed by integrodifferential 
operators. Solutions are defined with the
help of symmetric nonlocal bilinear forms. Throughout this 
work, our main emphasis is on
operators with general, possibly singular, measurable kernels. We obtain 
regularity results which are robust with respect to the differentiability order 
of the equation. Furthermore, we provide a general tool for the derivation of 
H\"{o}lder a-priori estimates from the weak Harnack inequality. This tool is 
applicable for several local and nonlocal, linear and nonlinear problems on 
metric spaces. Another aim of this work is to provide comparability results 
for nonlocal quadratic forms. 
\end{abstract}

\maketitle

\tableofcontents
\addtocontents{toc}{\protect\setcounter{tocdepth}{1}}

\section{Introduction}\label{sec:intro}

The aim of this work is to develop a local regularity theory for general 
nonlocal operators. The main focus is on operators that are defined through 
families of measures, which might be singular. The main question that we ask 
is the following. Given a
function $u:\R^d \to \R$ satisfying 
\begin{align}\label{eq:def-operator_intro} 
\lim\limits_{\eps \to 0+} \il_{\R^d\setminus B_{\eps}(x)} \big( u(y) - u(x) 
\big) \mu(x,\d y) = f(x) \quad (x 
\in D)
\,,
\end{align}
which properties of $u$ can be deduced
in the interior of $D$? Here $D \subset \R^d$ is a bounded open set and
the family $(\mu(x,\cdot))_{x \in D}$ of measures satisfies some
assumptions
to be discussed later in detail. The measures $\mu(x,\cdot)$ are assumed to 
have a singularity for sets $A \subset \R^d$ with $x \in \overline{A}$. As a 
result, the operators of the form \eqref{eq:def-operator_intro}  are not 
bounded integral operators but integrodifferential operators. For this reason we 
are able to prove regularity results which resemble results for differential 
operators. One aim of this work is to address an important conjecture 
in this field:

{\bf Conjecture:} {\it Assume $\mu(x,\d y)$ is uniformly (w.r.t. the 
variable $x$)  
comparable on small scales (w.r.t. the variable $y$) to $\nu^\alpha (\d 
y-\{x\})$ for some $\alpha$-stable measure 
$\nu^\alpha$ and 
\begin{align*}
 \inf\limits_{\xi \in \mathbb{S}^{d-1}} \il_{B_1} |\langle h, \xi \rangle 
|^2 \nu^\alpha(\d h) > 0
\end{align*}
for some $\alpha \in (0,2)$. Then solutions to 
\eqref{eq:def-operator_intro} satisfy uniform Hölder regularity estimates in the 
interior of $D$.}

This conjecture has received significant attention over the last years and we 
give a small 
overview of results below. Note that, assuming comparability of measures rather 
than of corresponding 
densities allows for a much wider class of cases that can be treated. In this 
work we provide a structural approach to this 
problem. We give an affirmative answer if $\mu(x, \cdot)$ is absolutely 
continuous on $\R^d$ or on sufficiently many subspaces. Note that it is well 
known how to treat functions $f$ in \eqref{eq:def-operator_intro}. Thus we will 
concentrate on the case $f = 0$.

In order to approach the question raised above, we need to establish the 
following results:
\begin{itemize}
 \item weak Harnack inequality,
 \item implications of the weak Harnack inequality,
 \item comparability results for nonlocal quadratic forms.
\end{itemize}
The last topic needs to be included because our concept of solutions 
involves quadratic forms related to $\mu(x,\d y)$. We present the main results 
in \autoref{subsec:intro_weak_harnack}, \autoref{subsec:intro_reg-estimates}, 
and in  \autoref{subsec:intro_comp}. The following two subsections are devoted 
to the set-up and our main assumptions. 

\subsection{Function spaces}\label{subsec:spaces}

Before we can formulate the first result we need to set up quadratic forms and 
function spaces. Let $\mu=(\mu(x,\cdot) )_{x \in \R^d}$ be a family of measures 
on $\R^d$ which 
is symmetric in the sense that for every set $A \times B \subset \R^d 
\times \R^d \setminus 
\operatorname{diag}$  
\begin{align}\label{eq:mu-symmetry}
 \int\limits_A \int\limits_B  \mu(x,\d y) \, \d x =  \int\limits_B
\int\limits_A 
\mu(x,\d y) \, \d x \,.
\end{align}
We furthermore require 
\begin{align}\label{eq:mu-integrability}
 \sup\limits_{x\in\R^d} \int\limits_{\R^d} \min \left( |x-y|^2,1 \right) 
\mu(x,\d y) < +\infty \,. 
\end{align}

\begin{example}\label{ex:def_mu-alpha} An important example satisfying the 
above 
conditions is given by
\begin{align}\label{eq:def_mu-alpha}
\mu_\alpha (x,\d y) = (2{-}\alpha) |x-y|^{-d-\alpha} \d y \qquad (0 <
\alpha < 2) \,. 
\end{align}
The choice of the factor $(2-\alpha)$ will be discussed below in detail, 
see \autoref{subsec:main-assum} and \autoref{sec:frac_laplace}.  
\end{example}

For a given 
family
$\mu$ and a real number $\alpha \in (0,2)$
we consider the following quadratic forms on $L^2(D) \times L^2(D)$, where
$D\subset \R^d$ is some open set:                                   
\begin{align}
\cE^\mu_D(u,u) &= \il_D \il_D (u(y)-u(x))^2 \mu (x,\d y)\,\d x \,.
\label{eq:Ekd}
\end{align}

We denote by $H^{\alpha/2}(\R^d)$ the usual Sobolev 
space of fractional order
$\alpha/2 \in (0,1)$ with the norm 
\begin{equation}\label{eq:Sobolevnorm}
\|u\|_{H^{\alpha/2}(\R^d)} = \left( \|u\|^2_{L^2(\R^d)} + 
\cE^{\mu_\alpha}_{\R^d}(u,u)
\right)^{1/2} \,.
\end{equation}
If $D \subset \R^d$ is
open and bounded, then by $H_D^{\alpha/2}=H_D^{\alpha/2}(\R^d)$ we denote the 
Banach space of functions from $H^{\alpha/2}(\R^d)$ which are zero almost 
everywhere on $D^c$. $H^{\alpha/2}(D)$ shall be the space of functions $u \in 
L^2(D)$ for which 
\begin{align*}
 \|u\|^2_{H^{\alpha/2}(D)} = \|u\|^2_{L^2(D)} + \int\limits_{D} \int\limits_{D} 
\big(u(y) - u(x) \big)^2 \mu_\alpha(x, \d y) \d x
\end{align*}
is finite. 
Note that, for domains $D$ with a Lipschitz boundary,  
$H_D^{\alpha/2}(\R^d)$ can be identified with the closure of $C_c^\infty (D)$ 
with respect to the norm of $H^{\alpha/2}(D)$. In general, these two objects 
might be different, though. By 
$V_D^{\alpha/2}=V_D^{\alpha/2}(\R^d)$ we denote the 
space of all measurable
functions $u:\R^d \to \R$ for which the quantity
\begin{align}
 \int_D \int_{\R^d} \frac{\big(u(y){-}u(x)\big)^2}{|x{-}y|^{d+\alpha}} \d x \d
y \end{align}
is finite, which implies finiteness of the quantity $ \int_{\R^d}
\frac{u(x)^2}{(1+|x|)^{d+\alpha}} \d x$. The function space $V_D^{\alpha/2}$ is
a Hilbert space with the scalar product
\begin{align}\label{eq:scalar_product_V}
 (u,v)_{V_D^{\alpha/2}} =  \int_{\R^d} \frac{u(x) v(x)}{(1+|x|)^{d+\alpha}}
\d x + \int_D \int_{\R^d}
\frac{\big(u(y){-}u(x)\big)\big(u(v){-}v(x)\big)}{|x{-}y|^{d+\alpha}} \d x \d y
\,.
\end{align}
The proof is similar to the one of \cite[Lemma 2.3]{FKV13} and the
one of \cite[Proposition 3.1]{DRV14}. If the scalar product 
\eqref{eq:scalar_product_V} is defined with the expression replaced 
$\int_{\R^d} \frac{u(x) v(x)}{(1+|x|)^{d+\alpha}}$ by
$\int_{D} u(x) v(x) \d x$, then the Hilbert space is identical. The 
following continuous embeddings trivially hold true:
\begin{align*}
H^{\alpha/2}_D(\R^d) \hookrightarrow H^{\alpha/2}(\R^d) \hookrightarrow 
V^{\alpha/2}_D(\R^d) \,.
\end{align*}

We make use of function spaces generated by general $\mu$ in the 
same way as above. Let $H^{\mu}(\R^d)$ be the vector space of functions $u \in 
L^2(\R^d)$ such that $\cE^\mu(u,u)=\cE^\mu_{\R^d}(u,u)$ is finite. If $D 
\subset \R^d$ is open and bounded, then by $H_D^{\mu}=H_D^{\mu}(\R^d)$ we 
denote 
the space of functions from $H^{\mu}(\R^d)$ which are zero almost 
everywhere on $D^c$. By $V_D^{\mu}=V_D^{\mu}(\R^d)$ we denote the space of all 
measurable
functions $u:\R^d \to \R$ for which the quantity
\begin{align}
 \int_D \int_{\R^d} \big(u(y){-}u(x)\big)^2 \mu(x,\d y) \d 
x \end{align}
is finite. Now we are in a position to present and discuss our main results.

\subsection{Main Assumptions}\label{subsec:main-assum} 

Let us formulate our main assumptions on $(\mu(x,\cdot))_{x \in D}$. Given 
$\alpha \in (0,2)$ and 
$A \geq
1$, the following condition is an analog
of \eqref{eq:assum_loc_comp} for nonlocal energy forms:
\begin{align} 
\label{eq:assum_comp}\tag{A} 
\begin{split}
& \text{ For every ball } B_\rho(x_0) \text{ with } \rho \in (0,1), x_0 \in B_1 
\text{
and every } v \in H^{\alpha/2}(B_\rho(x_0)): \\
& 
\qquad\qquad A^{-1} \, \cE^\mu_{B_\rho(x_0)}(v,v) \leq 
\cE^{\mu_\alpha}_{B_\rho(x_0)}(v,v)
\leq A
\,
\cE^\mu_{B_\rho(x_0)}(v,v) \,. 
\end{split}
\end{align}
Condition \eqref{eq:assum_comp} says that, locally in the unit ball, the 
energies $\cE^\mu$ and $\cE^{\mu_\alpha}$ are comparable on 
every scale. Note that this does not imply pointwise comparability of the 
densities of 
$\mu$ and $\mu_\alpha$. We also need to
assume the existence of cut-off functions. Let $\alpha \in (0,2)$ and $B \geq
1$. 
\begin{align} \label{eq:assum_cutoff}\tag{B}
\begin{split}
&\text{ For } 0 < \rho \leq R \leq 1 \text{ and } x_0 \in B_1 \text{ 
there is a nonnegative measurable function }
\\ 
& \tau : \R^d \to \R \text{ with }
\operatorname{supp}(\tau) \subset \overline{B_{R+\rho}(x_0)}, \tau(x) \equiv
1 \text{ on } B_{R}(x_0)\,, \|\tau\|_\infty \leq 1\,, \text{ and } \\
& \qquad \sup\limits_{x \in \R^d} \; \il_{\R^d} \big(\tau(y) - \tau(x)\big)^2  
\mu(x,\d y) \leq B \rho^{-\alpha} \,.
\end{split}
\end{align}

In most of the cases \eqref{eq:assum_cutoff} does not impose an additional 
restriction because the standard cut-off function $\tau(x) =
\max(0,1+\min(0,\frac{R-|x-x_0|}{\rho}))$ is an appropriate choice. It is an 
interesting
question whether, under assumptions
\eqref{eq:mu-symmetry}, \eqref{eq:mu-integrability} and  \eqref{eq:assum_comp},
this standard choice would be
possible in \eqref{eq:assum_cutoff}. Note that, condition 
\eqref{eq:assum_cutoff} becomes $|\nabla \tau|^2 \leq B \rho^{-2}$ when $\alpha 
\to 2-$ and $\mu(x,\d y)$ is as in \autoref{ex:def_mu-alpha}. 

For every $\alpha \in (0,2)$, the family of measures $\mu_\alpha$ given in 
\autoref{ex:def_mu-alpha} satisfies the above conditions for some constants 
$A,B 
\geq 1$. The normalizing constant $(2-\alpha)$ in the definition of 
$\mu_\alpha$ has the effect that the constants $A, B \geq 1$ can be chosen 
independently of $\alpha$ for $\alpha \to 2-$. Since in this work we do not 
care about the behavior 
of constants for $\alpha \to 0+$, in our examples we will use factors of the 
form $2-\alpha$. Let us look at more examples.

\begin{example}\label{exa:sum_comp_stable}
Assume $0 < \beta \leq \alpha < 2$. Let $f,g:\R^d \to [1,2]$ be measurable 
and symmetric functions. Set 
\begin{align*}
\mu(x, \d y) = f(x,y) \mu_\alpha(x, \d y) + g(x,y) \mu_\beta(x, \d y) \,. 
\end{align*}
Then $\mu$ satisfies \eqref{eq:mu-symmetry}, 
\eqref{eq:mu-integrability}, \eqref{eq:assum_comp}, and \eqref{eq:assum_cutoff}
with exponent $\alpha$. This simply follows from 
\[\frac{1}{|x-y|^{d+\alpha}} \leq \frac{1}{|x-y|^{d+\beta}} + 
\frac{1}{|x-y|^{d+\alpha}} \leq \frac{2}{|x-y|^{d+\alpha}} \qquad (x,y 
\in B_1(x_0), x_0 \in \R^d)\,. \]
For the verification of \eqref{eq:assum_cutoff} we may choose the standard 
Lipschitz-continuous cutoff function.
\end{example}

Here is an example with some kernels which are not rotationally symmetric. 
\begin{example}\label{exa:stable_cone}
Assume $\alpha_0 \in (0,2)$, $0 < \lambda < \Lambda $, $v \in S^{d-1}$ and
$\theta \in [0,1)$. Set $M= \{h \in \R^d | \, \vert \langle
\tfrac{h}{|h|}, v \rangle \vert \geq \theta \}$. Let $k:\R^d \times \R^d \to
[0,\infty]$ be any measurable function satisfying 
\begin{align}\label{eq:assum_k_simple}
  \lambda \mathbbm{1}_{M}(x-y) \, \frac{(2{-}\alpha)}{|x-y|^{d+\alpha}} \leq
k(x,y) \leq \Lambda \, \frac{(2{-}\alpha)}{|x-y|^{d+\alpha}} 
\end{align}
for some $\alpha \in [\alpha_0,2)$ and for almost every $x,y \in \R^d$. Set
$\mu(x,\d y) = k(x,y) \d y$. Then, as we will prove, there are $A \geq 1, B 
\geq 
1$, independent of $\alpha$, such that \eqref{eq:assum_comp} and 
\eqref{eq:assum_cutoff} hold. 
\end{example}

The following example of a family of measures falls into our framework. Note 
that the measures do not possess a density with respect to the $d$-dimensional 
Lebesgue measure.

\begin{example} \label{exa:axes} Assume $\alpha_0 \in (0,2)$, $\alpha_0 \leq 
\alpha < 2$. Set 
\begin{equation}\label{eq:exa_axes}
	\mu(x,dy) = (2-\alpha) \sum_{i=1}^d
\left[|x_i-y_i|^{-1-\alpha}\d y_i\prod_{j\not=i}\delta_{\{x_j\}}(\d 
y_j)\right]\,.
\end{equation}
Again, as we will prove, there are $A \geq 1, B \geq 
1$, independent of $\alpha$, such that \eqref{eq:assum_comp} and 
\eqref{eq:assum_cutoff} hold. Note that $\mu(x,A)=0$ for every set $A$ which 
has an empty intersection with any of the $d$ lines $\{x+t e_i| t 
\in \R\}$.
\end{example}

\medskip

Let us now formulate our results.

\subsection{The Weak Harnack Inequality}\label{subsec:intro_weak_harnack} 

Given functions 
$u,v: \R^d \to \R$ we define the quantity 
\begin{align} 
\cE^{\mu}(u,v) = \iil_{\R^d \R^d} 
\big(u(y)-u(x)\big)\big(v(y)-v(x)\big)
\mu(x,\d
y) \, \d x \,,
\end{align}
if it is finite. We write $\cE$ instead of $\cE^{\mu}$ when it is clear
resp. irrelevant which measure $\mu$ is used. One aim of this work is to study
properties of functions $u$
satisfying 
$\cE(u,\phi) \geq 0$ for every nonnegative test function $\phi$.  Note that 
$\cE^{\mu}(u,\phi)$ is finite for $u \in V_D^\mu$, $\phi \in H^\mu_D(\R^d)$ 
for any open set $D \subset \R^d$. This follows from the definition of these 
function spaces, the Cauchy-Schwarz inequality and the following decomposition:
\begin{align*}
 \cE^{\mu}(u,\phi) &= \iil_{D D} \big(u(y)-u(x)\big)\big(\phi(y)-\phi(x)\big)
\mu(x,\d y) \, \d x \\
&\quad + 2 \iil_{D D^c} \big(u(y)-u(x)\big)\big(\phi(y)-\phi(x)\big)
\mu(x,\d y) \, \d x \,. 
\end{align*}

Here is our first main result.

\begin{theorem}[Weak Harnack Inequality] \label{theo:weak_harnack_nonlocal}
Assume $0 < \alpha_0 < 2$ and $A\geq 1, B \geq 1$. Let $\mu$ satisfy
\eqref{eq:assum_comp}, \eqref{eq:assum_cutoff} for some $\alpha \in [\alpha_0,
2)$. Assume 
$f \in L^{q/\alpha}(B_1)$ for some $q>d$. Let $u \in 
V_{B_{1}}^\mu (\R^d) $, $u \geq
0$ in $B_1$, satisfy $\cE^\mu (u,\phi) \geq (f,\phi)$ for every 
nonnegative $\phi \in H^\mu_{B_1}(\R^d)$.
Then
\begin{align}\label{eq:weak_harnack_nonlocal}  \inf\limits_{B_{\frac14}} u 
\geq c \big( \fint\limits_{B_{\frac12}}
u(x)^{p_0} \,
\d x \big)^{1/p_0} - \sup\limits_{x \in B_{\frac{15}{16}}} \int\limits_{\R^d
\setminus B_1} u^-(z) \mu(x,\d z) - \|f\|_{L^{q/\alpha}(B_{\frac{15}{16}})} \,,
\end{align} 
with constants $p_0, c \in (0,1)$  depending only on $d,
\alpha_0, A, B$. In particular, $p_0$ and $c$ do not depend on
$\alpha$.
\end{theorem}

Note that, below we explain a local counterpart to this result, which relates 
to the limit $\alpha \to 2-$, cf. \autoref{theo:weak_harnack_local}.

\begin{remark}
It is remarkable that \eqref{eq:assum_comp} and \eqref{eq:assum_cutoff} do not 
imply a strong formulation of the Harnack inequality. Both, \autoref{exa:axes} 
and \autoref{exa:stable_cone} provide cases in which the classical strong 
formulation fails. See the discussion in \cite[Appendix A.1]{KRS14} and the 
concrete examples in \cite[p. 148]{BoSz05} and \cite[Sec. 3]{BaCh10}. The 
nonlocal term, i.e. the integral of $u^-$ in \eqref{eq:weak_harnack_nonlocal} 
is unavoidable since we do not assume nonnegativity of $u$ in all of $\R^d$.
\end{remark}

\subsection{Regularity estimates}\label{subsec:intro_reg-estimates}

A separate aim of our work is to provide consequences of the (weak) Harnack 
inequality. Before we explain this in a more abstract fashion let us formulate 
a regularity result, which will be derived from 
\autoref{theo:weak_harnack_nonlocal} and which is one of the main results of 
this work. We need an additional mild assumption on the decay of
the kernels considered. 

Given $\alpha \in (0,2)$ we assume that for some constants $\chi > 1$, $C 
\geq 1$ 
\begin{align}\label{eq:assum_largejumps}\tag{D}
\mu(x,\R^d \setminus B_{r 2^j}(x)) \leq Cr^{-\alpha} \chi^{-j} \qquad (x \in 
B_1, 0 < r \leq 1, j \in \N_0) \,.
\end{align}
Condition \eqref{eq:assum_largejumps} rules out kernels with very heavy
tails for large values of $|x-y|$. For example, $\mu$ given by $\mu(x,\d y) =
k(x,y)\d y$ with $k(x,y)=|x-y|^{-d-1} +
|x-y|^{-d}\ln(2+|x-y|)^{-2}$ does not satisfy \eqref{eq:assum_largejumps}.

Here is our main regularity result.

\begin{theorem}\label{theo:reg_result_nonlocal} Let $\alpha_0
\in (0,2), \gamma > 0$ and $A\geq 1, B \geq 1$. Let $\mu$ satisfy
\eqref{eq:assum_comp}, \eqref{eq:assum_cutoff} and \eqref{eq:assum_largejumps}
for some $\alpha \in [\alpha_0, 2)$. Assume $u \in V^\mu (B_{1})$ satisfies 
$\cE(u,\phi) = 0$ for some $x_0 \in
\R^n$ and every $\phi \in H^\mu_{B_1}(\R^d)$. Then 
the following H\"{o}lder estimate holds for almost every $x,y \in
B_{\frac{1}{2}}$:
\begin{align}\label{eq:hoelder_final_result} 
|u(x) - u(y)| \leq c \|u\|_\infty |x-y|^\beta \,,
\end{align}
where $c \geq 1$ and $\beta \in (0,1)$ are constants which depend only on
$d, \alpha_0, A, B,C, \gamma$. In particular, $c$ and $\beta$ do not
depend on $\alpha$.
\end{theorem}

This result contrasts the corresponding result 
for differential operators, see \autoref{theo:reg_result_local} below.

The main tool for the proof of \autoref{theo:reg_result_nonlocal} is the
weak Harnack inequality, \autoref{theo:weak_harnack_nonlocal}. The Harnack
inequality itself is an interesting object of study for nonlocal operators. In
\autoref{sec:frac_laplace}  we have explained different formulations of the
Harnack inequality for nonlocal operators satisfying a maximum
principle.  A separate aim of this article is to prove a
general tool that allows to deduce regularity estimates from the Harnack
inequality for nonlocal operators. This step was subject to discussion of many
recent articles in the field. We choose the set-up of a metric measure
space so that this tool can be of future use in different contexts. 

In the first decades after publication the Harnack inequality itself did not
attract as much of attention as the resulting convergence theorems.
This changed when J. Moser in 1961 showed that the inequality itself
leads to a-priori estimates in H\"older spaces. His result can be 
formulated in a metric measure space $(X,d,m)$ as follows. For $r>0$, $x \in
X$, set $B_r(x) =\{ y \in X | d(y,x) < r \}$.

For every $x\in X$ and $r>0$ let $\cS_{x,r}$ denote a family of measurable 
functions on $X$ satisfying the following
conditions:
\begin{align*}
r>0, u\in \cS_{x,r}, a \in \R \quad &\Rightarrow \quad a u \in
\cS_{x,r} , (u+1) \in \cS_{x,r} \,, \\
B_r(x) \subset B_s(y) \quad &\Rightarrow \quad \cS_{y,s}\subset \cS_{x,r} \,. 
\end{align*}
An example for $\cS_{x,r}$ is given by the set of all
functions $u:\R^d \to \R$ satisfying some (possibly 
nonlinear) appropriate partial differential or integro-differential equation in 
a ball $B_r(x)$. 

\begin{theorem}[compare \cite{Mos61}] \label{theo:harnack_hoelder} Assume $X$ 
is 
separable. Let $x_0 \in
X$ and $\mathcal{S}_{x,r}$ be as above. Assume that there is $c\geq
1$ such
that for $r>0$,
\begin{align}\label{eq:local-hoelder-harnack-ass}
\big(u \in \mathcal{S}_{x_0,r}) \wedge \big(u \geq 0 \mbox{ in } B_r(x_0) \big)
\quad \text{ implies } \quad \sup\limits_{x \in B_{\frac{r}{2}}(x_0)} u
\leq c
\inf\limits_{x \in B_{\frac{r}{2}}(x_0)} u \,. 
\end{align}
Then there exist $\beta \in (0,1)$ such that for $r>0$, $u\in
\mathcal{S}_{x_0,r}$ and almost every $x
\in B_r(x_0)$
\begin{align*} 
|u(x) - u(x_0)| \leq 3 \|u-u(x_0)\|_\infty \Big( \frac{d(x,x_0)}{r} \Big)^\beta
\, .
\end{align*}
\end{theorem}

Recall that '$\sup$' denotes the essential
supremum and  '$\inf$' the essential infimum. With the help of this theorem,
regularity estimates can be established for various linear 
and nonlinear differential equations, see \cite{GiTr83}. One aim of this
article is to show that \eqref{eq:local-hoelder-harnack-ass} can be relaxed
significantly by allowing some global terms of $u$ to show up in the Harnack
inequality. Already in \autoref{sec:frac_laplace} we have seen that they 
naturally
appear. 

For $x \in X, r>0$ let $\nu_{x,r}$ be a
measure on $\cB(X \setminus \{x\})$, which is finite on all sets $M$ with
$\dist(\{x\},M)>0$. We assume that for some $c\geq 1$,
$\chi > 1$, and for every
$j\in\N_0$, $x\in X$ and $0<r\leq 1$
\begin{align}\label{eq:nu-decay-simple}
\nu_{x,r}(X \setminus B_{r 2^j} (x)) \leq  c \chi^{-j} \,.
\end{align}
We further assume that, given $K>1$ there is $c \geq 1$ such that for $0 < r
\leq R \leq Kr$, $x \in X$, $M \subset X \setminus B_r(x)$ 
\begin{align}\label{eq:nu_r-depend-assum} 
\nu_{x,R}(M) \leq c \, \nu_{x,r}(M)  \,.
\end{align}
Conditions \eqref{eq:nu-decay-simple} and \eqref{eq:nu_r-depend-assum} will
trivially hold true in the applications that are of importance to us.

\begin{example}\label{exa:standard_nu-xr} Let $\alpha \in (0,2)$. For $x \in
\R^d$, $r>0$
and $A \in \cB(\R^d \setminus \{x\})$ set
\begin{align}
\nu_{x,r}(A) = r^{\alpha} \mu_\alpha(x,A) = r^{\alpha}
\alpha(2{-}\alpha) \il_A |x-y|^{-d-\alpha}  \d y \,. 
\end{align}
Then $\nu_{x,r}$ satisfies conditions \eqref{eq:nu-decay-simple},
\eqref{eq:nu_r-depend-assum}.
\end{example}

In \autoref{sec:harnackimplieshoelder} we discuss this condition in detail. A
standard example for us is \autoref{exa:standard_nu-xr}. The following result 
extends \autoref{theo:harnack_hoelder} to
situations with nonlocal terms. It is an important tool in the theory of
nonlocal operators. 

\begin{theorem}\label{theo:hh-nonloc-gen-k} Let $x_0 \in X$, $r_0>0$ and
$\lambda >1, \sigma>1, \theta > 1$. Let $\cS_{x,r}$ and $\nu_{x,r}$ be as
above. Assume that conditions \eqref{eq:nu-decay-simple},
\eqref{eq:nu_r-depend-assum} are satisfied. Assume
that there is $c \geq 1$ such that
for $0 < r \leq r_0$,
\begin{align}\label{eq:hoelder-harnack-lem}
\left.
\begin{cases}
\big(u \in \cS_{x_0,r}) \wedge \big(u \geq 0 \text{ in } B_r(x_0) \big) \,,
\\
\quad \Rightarrow \quad
\Big( \fint\limits_{B_{\frac{r}{\lambda}}(x_0)} u(x)^p m(\d x) \Big)^{1/p}
\leq
c
\inf\limits_{x \in B_{\frac{r}{\theta}}(x_0)} u + c  \sup\limits_{x\in
B_{\frac{r}{\sigma}}(x_0)} \il_{X} u^-(z) \nu_{x,r} (\d z) \,.
\end{cases} \right\}
\end{align}
Then there exist $\beta \in (0,1)$ such that for $0 < r \leq r_0$,
$u\in \cS_{x_0,r}$ 
\begin{align} \label{eq:hoelder-asser-abstr-lem}
\osc\limits_{B_\rho(x_0)} u \leq 2\theta^\beta \|u\|_\infty
\left(\frac{\rho}{r}\right)^\beta \qquad (0<\rho \leq r) \,,
\end{align}
where $\osc\limits_{M} u := \sup\limits_M u - \inf\limits_M u$ for $M \subset
X$.
\end{theorem}

Note that, in \autoref{lem:nu-xr_prop} we provide several conditions that are
equivalent to \eqref{eq:nu-decay-simple}.

\subsection{Comparability of nonlocal quadratic forms}\label{subsec:intro_comp}

With regard to \autoref{theo:reg_result_nonlocal} one major problem is to 
provide conditions on $\mu$ which imply \eqref{eq:assum_comp}. Let us 
formulate our results in this direction. 

Since $\mu=(\mu(x,\cdot))_{x \in \R^d}$ is a family of measures we need to 
impose a condition that fixes a uniform behavior of $\mu$ with respect to 
$x$. In our setup this condition implies that 
the integrodifferential operator from \eqref{eq:def-operator_intro} is 
comparable to a translation invariant operator - most often the generator of an 
$\alpha$- stable process. We assume that there are measures $\loB$ and $\upB$ 
such that
\begin{equation}\label{eq.lo-up-meas}\tag{T}
\int f(x,x+z) \loB(dz) \leq \int f(x,y) \mu(x,dy) \leq \int f(x,x+z) \upB(dz)
\end{equation}
for every measurable function $f:\R^d\to [0,\infty]$ and every $x \in \R^d$.
For a measure $\nu$ on $\R^d$ such that $\nu(\{0\})=0$ and a~set $B\subset
\R^d$ we define, abusing the previous notation slightly,
\begin{equation}
\cE^{\nu}_B(u,v) = \int_B \int_{\R^d} \big(u(x)-u(x+z)\big)\big(v(x)-v(x+z)\big)
\1_B(x+z) \, \nu(dz)\,dx.
\end{equation}
Note that \eqref{eq.lo-up-meas} implies for every $u \in L^2(B)$
\[
\cE^{\nu_*}_B(u,u) \leq \cE^{\mu}_B(u,u) \leq \cE^{\nu^*}_B(u,u)\,.
\]

Let $\overline{\nu}(A) = \nu(-A)$. It is easy to check that $\cE^{\nu} =
\cE^{\frac{\nu+\overline{\nu}}{2}}$.
Hence we may and do assume that the measures $\nu_*$, $\nu^*$ are symmetric,
i.e., $\nu_*(A) = \nu_*(-A)$ and 
$\nu^*(A) = \nu^*(-A)$.

We say that a measure $\nu$ on $\cB(\R^d)$ satisfies the upper-bound 
assumption \eqref{eq:cond_U} if for some $C_U > 0$
\begin{align}\label{eq:cond_U}\tag{U}
\int_{\R^d} (r\wedge |z|)^2 \nu(\d z) \leq C_U r^{2-\alpha} \qquad (0 < r 
\leq 1)\,.
\end{align}

We say that a measure $\nu$ on $\cB(\R^d)$ satisfies the scaling assumption 
\eqref{eq:cond_scaling} if for some $a>1$
\begin{align}\label{eq:cond_scaling}\tag{S}
\int_{\R^d} f(y) \nu(\d y) = a^{-\alpha} \int_{\R^d} f(ay) \nu(\d y),
\end{align}
for every measurable function $f:\R^d\to [0,\infty]$ with $\supp f \subset 
B_1$. For a linear subspace $E\subset \R^d$, let $H_{E}$ denote the
$\dim(E)$--dimensional Hausdorff measure supported on $E$.

We say that a measure $\nu$ on $\cB(\R^d)$ satisfies the nondegeneracy 
assumption \eqref{eq:cond_nondeg} if for some $n \in \{1, \ldots, d\}$ 
\begin{align}\label{eq:cond_nondeg}\tag{ND}
\begin{split}
&\nu = \sum_{k=1}^n f_k H_{E_k} \text{ for some linear subspaces } E_k\subset 
\R^d \text{ and densities } f_k \\ &\text{with } \lin(\cup_k E_k) = 
\R^d \text{ and } \int_{B_1} f_k dH_{E_k} > 0 \text{ for } k=1,\ldots,n \,.
\end{split}
\end{align}

Here is our result on local comparability of nonlocal energy forms:

\begin{theorem}\label{theo:U1L1implyA} Let $\mu=(\mu(x,\cdot) )_{x\in \R^d}$ be 
a family of measures on $\cB(\R^d)$ satisfying \eqref{eq:mu-symmetry}. Assume 
that there exist measures $\loB$ and $\upB$ for which
\eqref{eq.lo-up-meas} and
\eqref{eq:cond_U} hold with $\alpha_0 \in (0,2)$ and  $C_U>0$. Assume 
that $\loB$ satisfies \eqref{eq:cond_nondeg} and each measure 
$f_k H_{E_k}$ satisfies \eqref{eq:cond_scaling} for some fixed $a > 1$.
 Then there are $A \geq 1$, $B \geq 1$
such that \eqref{eq:assum_comp} and \eqref{eq:assum_cutoff} hold. One can 
choose $B=4C_U$ but the constant $A$ depends also on $a$, the measure $\loB$ 
and on $\alpha_0$.

The result is robust in the following sense: If 
$\mu^{\alpha}=(\mu^{\alpha}(x,\cdot) )_{x\in \R^d}$ satisfies
\eqref{eq:mu-symmetry}
and  \eqref{eq.lo-up-meas} with measures $(\loB)^\alpha$ and $(\upB)^\alpha$, 
$\alpha_0 \leq \alpha < 2$, that are defined with the help of $\loB$ and $\upB$ 
as in \autoref{def:nualpha}, then \eqref{eq:assum_comp} holds with a constant 
$A$ independent of $\alpha\in [\alpha_0,2)$.
\end{theorem}

\subsection{Related results}{\ }

It is instructive to compare or results with two key results for 
differential operators in 
divergence form. Let $(A(x))_{x
\in \R^d}$ be a family of $d \times d $-matrices. Given a
subset $D \subset \R^d$ we introduce a bilinear form $\cA_D$ by
$\mathcal{A}_D(u,v)= \int\limits_D ( \nabla u(x), A(x) \nabla u(x) ) \, \d x$ 
for $u$ and $v$ from the Sobolev space $H^1(D)$. Instead of $\cA_{\R^d}$ we 
write $\cA$. The following theorem
is at the heart of the theory named after E. DeGiorgi, J. Moser and J. Nash, see
\cite[Ch. 8.8-8.9]{GiTr83}:

\begin{theorem}[Weak Harnack Inequality] \label{theo:weak_harnack_local}
Let $\Lambda > 1$. Assume that for all balls $B \subset B_1$ and all functions 
$v \in H^1(B)$
\begin{align} \label{eq:assum_loc_comp} \tag{A'} 
 \Lambda^{-1} \, \cA_B(u,u) \leq \int\limits_B |\nabla u|^2 \leq \Lambda \, 
\cA_B(u,u) \,. 
\end{align}
Assume $f \in L^{q/2}(B_1)$ for some $q>d$. Let $u \in
H^1(B_{1})$ satisfy $u \geq 0$ in $B_1$ and $\cA_{B_1}(u,\phi) \geq (f,\phi)$
for every nonnegative $\phi \in
H^1_0(B_1)$. Then 
\[   c \inf\limits_{B_{\frac14}} u \geq \big( \fint\limits_{B_{\frac12}}
u(x)^{p_0} \, \d x \big)^{1/p_0}  - \|f\|_{L^{q/2}(B_{\frac{15}{16}})}\,, \] 
with constants $p_0, c \in (0,1)$  depending only on $d$ and $\Lambda$. 
\end{theorem}

\begin{remark}
This by now classical result can be seen as the limit case of 
\autoref{theo:weak_harnack_nonlocal} for $\alpha \to 2-$. Condition 
\eqref{eq:assum_loc_comp} implies that the differential operator 
$\opn{div}(A(\cdot) \nabla u)$ is uniformly 
elliptic and obviously describes a limit situation of \eqref{eq:assum_comp}. 
One might object that the nonlocal term in 
\eqref{eq:weak_harnack_nonlocal} is unnatural but in fact, it is 
not. In \autoref{sec:frac_laplace} we explain this phenomenon in detail for the 
fractional Laplace operator.
\end{remark}

If $u$ is not only a supersolution but a solution in 
\autoref{theo:weak_harnack_local}, then one obtains a
classical Harnack inequality: $\sup_{B_{\frac14}} u \leq c \inf_{B_{\frac14}}
u$. Either one, the Harnack inequality and the weak Harnack inequality, imply
H\"{o}lder a-priori regularity estimates:

\begin{theorem}\label{theo:reg_result_local} Assume condition
\eqref{eq:assum_loc_comp} holds true. There
exist $c\geq 1$, $\beta \in (0,1)$ such that for every $u \in
H^{1}(B_{1})$
satisfying $\cA(u,\phi) = 0$ for every $\phi \in H^1_0(B_1)$
the following H\"{o}lder estimate holds for almost every $x,y \in
B_{\frac{1}{2}}$:
\begin{align}\label{eq:hoelder_final_result_local} 
|u(x) - u(y)| \leq c \|u\|_\infty |x-y|^\beta \, .
\end{align}
The constants $\beta, c$ depend only on $d$ and $\Lambda$.
\end{theorem}

After having recalled corresponding results for local differential operators, 
let us review some related results for nonlocal 
problems. Note that we restrict ourselves to nonlocal equations related 
to bilinear forms resp. distributional solutions.  

\autoref{theo:reg_result_nonlocal} has already been
proved under additional assumptions. If $\mu(x,\cdot)$
has a density $k(x,\cdot)$ which satisfies some isotropic lower bound, e.g. for
some  $c_0 > 0$, $\alpha \in (0,2)$  
\[\mu(x,\d y) = k(x,y) \d y, \quad k(x,y) \geq c_0 |x-y|^{-d-\alpha} 
\qquad (|x-y|\leq 1) \,,\]
then \autoref{theo:reg_result_nonlocal} is proved in resp. follows from the
works \cite{Kom95, BaLe02TAMS, ChKu03, CCV10}. In these works the constant
$c$ in \eqref{eq:hoelder_final_result} depends on $\alpha \in (0,2)$ with
$c(\alpha) \to +\infty$ for $\alpha \to 2-$. The current work follows the 
strategy laid out in \cite{Kas09} which, on the one hand, allows the constants 
to be independent of $\alpha$ for $\alpha \to 2-$ and, on the other hand, 
allows to treat general measures. See \cite{FeKa13} and 
\cite{KaSc14} for corresponding results in the parabolic case. 

The articles \cite{DiCKP14}, \cite{DiCKP15} study Hölder regularity estimates 
and Harnack inequalities for nonlinear equations. Moreover, the results 
therein provide boundedness of weak solutions. In \cite{DiCKP14}, 
\cite{DiCKP15} 
the measures $\mu(x, \d 
y)$ are assumed to be absolutely continuous with respect to the Lebesgue 
measure. Another difference to the present article is that our local 
regularity estimates require only local conditions on the data and on the 
operator. Note that our study of implications of 
(weak) Harnack 
inequalities in 
\autoref{sec:harnackimplieshoelder} allows for nonlinear 
problems in metric measure 
spaces and could be used to deduce the regularity results of 
\cite{DiCKP15} from results in \cite{DiCKP14}.

To our best knowledge there has been  no contribution addressing 
the question of comparability of quadratic nonlocal forms, cf. 
\autoref{sec:comparability}. This question becomes important when studying 
very irregular kernels as in \cite[Section 4]{Sil15}.

The conjecture mentioned in the beginning of the introduction has recently been 
established in the translation invariant 
case, i.e., when $\mu(x, \d y) = \nu^\alpha(\d y-\{x\})$ for some 
$\alpha$-stable measure $\nu^\alpha$, cf. \cite{RoSe15}. The methods of 
\cite{RoSe15} seem not to be applicable in the general case, though.

Related questions on nonlocal Dirichlet forms on metric measure spaces are 
currently investigated by several groups. We refer to the exposition 
in \cite{GHL14} for a discussion of results regarding the fundamental solution.

\subsection{Notation} Throughout this article, 
''inf'' denotes the essential infimum, ''sup'' the essential supremum. By 
$S^{d-1}=\{x\in\R^d |\, |x|=1\}$ we denote the unit sphere. We define the 
Fourier 
transform as an isometry of $L^2(\R^d)$ determined by
\[
 \hat{u}(\xi) = (2\pi)^{-d/2} \int_{\R^d} u(x)e^{-i\xi\cdot x}\,\d x, \quad u\in
L^1(\R^d)\cap L^2(\R^d).
\]

\subsection{Structure of the article}

The paper is organized as follows. In \autoref{sec:frac_laplace}
we study the Harnack inequality for the Laplace and the fractional 
Laplace operator. We explain how one can formulate a Harnack inequality without 
assuming the functions under consideration to be nonnegative. In 
\autoref{sec:prelims} we provide several auxiliary results and explain how
the inequality $\cE^\mu (u,\phi) \geq (f,\phi)$ is affected by rescaling the
family of measures $\mu$. In \autoref{sec:weak_solutions} we prove
\autoref{theo:weak_harnack_nonlocal} under assumptions \eqref{eq:assum_comp} and
\eqref{eq:assum_cutoff} adapting the approach by Moser to nonlocal bilinear
forms. \autoref{sec:regularity} provides the proof
of \autoref{theo:reg_result_nonlocal}. We first prove a general tool which 
allows to deduce regularity results from weak Harnack
inequalities, see \autoref{cor:hh-nonloc-gen-k}. Then 
\autoref{theo:reg_result_nonlocal} follows immediately. In 
\autoref{sec:comparability}
we study the question which conditions on $\mu$
are sufficient for conditions \eqref{eq:assum_comp} and
\eqref{eq:assum_cutoff} to hold true. In addition, we provide two 
examples of quite irregular kernels satisfying \eqref{eq:assum_comp} and
\eqref{eq:assum_cutoff}.

\section{Harnack inequalities for the Laplace and the fractional Laplace 
operator}\label{sec:frac_laplace}

We establish a formulation of the Harnack inequality which does not require
the functions to be nonnegative. This reformulation is especially
interesting for nonlocal problems but our formulation seems to be new even for
harmonic functions in the classical sense, see \autoref{theo:harnack_frac_new}.
For $\alpha\in (0,2)$ and $u\in C^2_c(\mathbb{R}^d)$ the fractional power of the
Laplacian can be defined as follows:
\begin{align}
\begin{split} \label{eq:def_frac-Laplace}
 \Delta^{\alpha/2} u(x) &= C_{\alpha,d} \lim\limits_{\varepsilon \to
0+} \!\!\int\limits_{|y-x|>\varepsilon}
\!\! \frac{u(y){-}u(x)}{|y{-}x|^{d+\alpha}} \,
\d y= \tfrac{C_{\alpha,d}}{2} \int\limits_{\R^d} \frac{u(x{+}h) {-} 2
u(x)
{+} u(x{-}h)}{|h|^{d+\alpha}} \, dh\,.
\end{split}
\end{align}
where $C_{\alpha,d} =
\frac{\Gamma((d+\alpha)/2)}{2^{-\alpha}\pi^{d/2}
|\Gamma(-\alpha/2)|}$. For later purposes we note that with some constant $c>0$
for every $\alpha \in (0,2)$
\begin{align}\label{eq:c-alpha-d_prop}
c\ \alpha (2{-}\alpha) \leq C_{\alpha,d} 
\leq \tfrac{\alpha (2{-}\alpha)}{c} \,.  
\end{align}
The use of the symbol $\Delta^{\alpha/2}$ and the term ``fractional Laplacian''
are justified because of
$\widehat{(-\Delta)^{\alpha/2} u} (\xi) = |\xi|^\alpha \widehat{u}(\xi)$ for
$\xi \in \R^d$ and $u \in C^\infty_c(\R^d)$. Note that we write
$\Delta^{\alpha/2} u$ instead of $- (-\Delta)^{\alpha/2} u$ which would be more
appropriate. The potential theory of these operators was initiated in
\cite{Rie38}. The
following Harnack inequality can be easily established using the corresponding
Poisson kernels.

\begin{theorem}\label{theo:harnack_frac_classic}
There is a constant $c\geq 1$ such that for $\alpha \in (0,2)$ and $u
\in C(\R^d)$ with
\begin{align}
 \Delta^{\alpha/2} u(x) &= 0 \qquad (x \in B_1) \,, \\
 u(x) &\geq 0 \qquad (x \in \mathbb{R}^d)\,,
\end{align}
the following inequality holds:
\begin{align*}
 u(x) \leq c u(y)\qquad (x,y \in B_{\frac12}) \,.
\end{align*}
\end{theorem}

Note that $\Delta^{\alpha/2} u(x) = 0$ at a point $x \in \Rd$ requires that
the integral in \eqref{eq:def_frac-Laplace} converges. Thus some additional
regularity of $u \in C(\R^d)$ is assumed implicitly. Since
$\Delta^{\alpha/2}$ allows for shifting and scaling, the result holds true for
$B_1, B_\frac12$ replaced by $B_R(x_0), B_\frac{R}{2}(x_0)$ with the same
constant $c$ for arbitrary $x_0 \in \R^d$ and $R>0$.

\autoref{theo:harnack_frac_classic} formulates the Harnack
inequality in the
standard way for nonlocal operators. The function $u$ is assumed to be
nonnegative in all of $\R^d$. In the following we discuss the necessity of this
assumption and possible alternatives. The following result proves that this
assumption cannot be dropped completely. 

\begin{theorem}\label{theo:noclassharnack}
Assume $\alpha \in (0,2)$. Then there exists a bounded function $u \in C(\R^d)$,
which is infinitely many times differentiable in
$B_1$ and satisfies
\begin{flalign*}
\Delta^{\alpha/2} u  (x) &= 0 \qquad (x \in B_1)\,, \\
u (x) &> 0\qquad (x \in B_1\setminus\{0\})\,, \\
u (0) &= 0 \,.
\end{flalign*}
Therefore, the classical local formulation of the Harnack inequality as well
as the local maximum principle fail for the operator $\Delta^{\alpha/2}$.
\end{theorem}

A complicated and lengthy proof can be found in
\cite{Kas07_SFB}. An elegant way to construct a function would be to
mollify $v(x) = (1-|x|^2)^{-1+\tfrac{\alpha}{2}}$ for $x \in B_1$. Here we
provide a short proof\footnote{We owe the idea to this proof to Wolfhard
Hansen.} which includes a helpful observation on radial functions. 

For an open set $D \subset \R^d$, $x \in D$, $0 < \alpha \leq 2$ and $v :\R^d
\to \R$ ($0<\alpha < 2$) resp. $v :\overline{D} \to \R$ ($\alpha = 2$) we write
\begin{align}
\begin{split}
H_\alpha(v|D)(x) &= \int_{y \notin D} P_\alpha(x,y) v(y) \, \d y =
\begin{cases}
\int\limits_{\R^d \setminus D} P_\alpha(x,y) v(y) \, \d y\, \quad & (0 <
\alpha < 2)\, \\ 
\int\limits_{\partial D} P_2(x,y) v(y) \, \d y\, \quad & (\alpha =2) \,.
\end{cases}
\end{split}
\end{align}

Note that for $R>0$ and $f:\R^d\setminus B_R(0) \to \R$ 
\begin{align*}
H_\alpha(f|B_R(0))(x) =
\begin{cases}  
f(x) \qquad &(|x|\geq R) \,, \\
c_\alpha (R^2-|x|^2)^{\alpha/2} \il_{|y|>R} 
 \frac{f(y)\,\d y}{(|y|^2-R^2)^{\alpha/2} |x-y|^{d} } \qquad &(|x|<R)\,,
\end{cases}
\end{align*}
where $c_\alpha = \pi^{-d/2-1} \Gamma(\tfrac{d}{2}) \sin\frac{\pi \alpha}{2}$.
For a function $\phi:[0,\infty) \to [0,\infty)$ we set
\[
 h^\phi_R := H_\alpha(\phi \circ | \cdot |\, |B_R(0)) \,.
\]

\begin{proposition}
For all $0<|x|<R$
\[
 h^\phi_R(x) = \frac{\sin \frac{\pi\alpha}{2}}{\pi} \int_0^\infty
\phi(\sqrt{R^2+s (R^2-|x|^2)})  \frac{ds}{(s+1)s^{\alpha/2}}.
\]
\end{proposition}
\begin{proof}
Let us fix $R>0$ and $x\in B_R(0)$. Using polar coordinates we obtain
\begin{equation}\label{eq:hphi}
h^\phi_R(x) = c_\alpha (R^2-|x|^2)^{\alpha/2} 
  \int_R^\infty \int_{\rho S^{d-1}} |x-y|^{-d} \,\sigma(\d y)\,
\frac{\phi(\rho)\,d\rho}{(\rho^2-R^2)^{\alpha/2}}.
\end{equation}
By the classical Poisson formula
\[
 \int_{S^{d-1}} \frac{1-|w|^2}{|w-y|^d}\,\sigma(\d y) = |S^{d-1}| \qquad
(|w|<1),
\]
hence
\begin{align*}
 \int_{\rho S^{d-1}} |x-y|^{-d} \,\sigma(\d y)
 &=
\rho^{-1} \int_{S^{d-1}} |\tfrac{x}{\rho}-y|^{-d} \,\sigma(\d y)
= \rho^{-1} |S^{d-1}| (1-\tfrac{|x|^2}{\rho^2})^{-1} \\
&= \frac{2\pi^{d/2}}{\Gamma(\frac{d}{2})} \frac{\rho}{\rho^2-|x|^2}.
\end{align*}
Plugging this into \eqref{eq:hphi} yields
\[
h^\phi_R(x) 
  = \frac{c_\alpha \pi^{d/2}}{\Gamma(\frac{d}{2})} (R^2-|x|^2)^{\alpha/2}
   \int_R^\infty  \frac{2\rho\phi(\rho)
\,d\rho}{(\rho^2-|x|^2)(\rho^2-R^2)^{\alpha/2}}.
\]
The simple substitution $s = (\rho^2-R^2)/(R^2-|x|^2)$ leads to
\[
\int_R^\infty  \frac{2\rho\phi(\rho)
\,d\rho}{(\rho^2-|x|^2)(\rho^2-R^2)^{\alpha/2}}
= \frac{1}{(R^2-|x|^2)^{\alpha/2}} \int_0^\infty \phi(\sqrt{R^2+s(R^2-|x|^2)})
\frac{ds}{(s+1)s^{\alpha/2}}.
\]
Thus the assertion follows.
\end{proof}

\autoref{theo:noclassharnack} now follows directly from the
following corollary. 

\begin{corollary}
Let $R>0$ and suppose that $\phi$ is decreasing on $[R,\infty)$ such
that
$\phi(s) < \phi(r)$ for some $R<r<s$. Then
\[
 h^\phi_R(x) < h^\phi_R(y), \qquad \text{whenever $0\leq x < y < R$.}
\]
In particular, $u:= h^\phi_R -  h^\phi_R(0)$ is a bounded function on $\R^d$
which
is $\alpha$-harmonic on $B_R(0)$ and satisfies $0=u(0) < u(y)$ for every $y\in
B_R(0)$.
\end{corollary}

In \autoref{theo:harnack_frac_classic} the function $u$ is assumed to be
nonnegative in all of $\R^d$. It is not plausible that the assertion should be
false for functions $u$ with small negative values at points far from the
origin. A similar question can be asked for classical harmonic functions. If $u$
is positive and large on a large part of $\partial B_1$, it should not matter
for the Harnack inequality on $B_\frac12$ if $u$ is negative with small absolute
values on a small part of $\partial B_1$. Another motivation for a
different formulation of the Harnack inequality is that
\autoref{theo:harnack_frac_classic} does not allow to use Moser's approach
to
regularity estimates, Theorem like \autoref{theo:harnack_hoelder} in a
straightforward manner.

Let us give a new formulation of the Harnack\footnotemark \ inequality that
does not
need any sign assumption on $u$. It is surprising that that this formulation
seems not to have been established since Harnack's textbook in
1887. We treat the classical local case $\alpha = 2$ together with the nonlocal
case $\alpha \in (0,2)$.

\footnotetext{The second author would like to use the opportunity to correct
an error in \cite{Kas07_survey} concerning the name Harnack. The correct name 
of the mathematician Harnack is Carl Gustav Axel Harnack. His renowned twin 
brother Carl Gustav Adolf carried the last name ``von Harnack'' after 
being granted the honor.}

\begin{theorem}\label{theo:harnack_frac_new}
(Harnack inequality for $\Delta^{\alpha/2}$, $0 < \alpha \leq
2$)\begin{enumerate}
\item[(1)] There is a constant $c\geq 1$ such that
for $0 < \alpha \leq 2$ and $u \in C(\R^d)$ satisfying
\begin{align}\label{eq:harmon_B1}
\Delta^{\alpha/2} u(x) &= 0 \qquad (x \in B_1) \,,
\end{align}
the following estimate holds for every $x,y \in B_{\frac12}$:
\begin{align}\label{eq:harnack_main_stable}
c\Big( u(y) - H_\alpha(u^+|B_{1})(y) \Big) \leq u(x) &\leq c\Big( u(y) +
H_\alpha(u^-|B_{1})(y) \Big) \,.
\end{align}
\item[(2)] There is a constant $c \geq 1$ such that for $0 < \alpha \leq
2$ and every function $u
\in C(\R^d)$, which satisfies \eqref{eq:harmon_B1} and is nonnegative in 
$B_1$, the following inequality holds for
every $x,y \in B_{\frac12}$:
\begin{align}\label{eq:harnack_alpha_pos}
u(x) &\leq c \big( u(y) + \alpha(2{-}\alpha) \!\!\il_{\R^d \setminus B_1}\!\!
\frac{u^-(z)}{|z|^{d+\alpha}} \, dz \big) \,.
\end{align}
\end{enumerate}
\end{theorem}

\begin{proof}[Proof of \autoref{theo:harnack_frac_new}] The
decomposition $u=u^+
- u^-$ and an application of \autoref{theo:harnack_frac_classic} gives 
\begin{align*}
 u(x) &= H_\alpha(u|B_{1})(x) \leq H(u^+|B_{1})(x) \leq c
H_\alpha(u^+|B_{1})(y) \\ &= c
H_\alpha(u|B_{1})(y) + c H_\alpha(u^-|B_{1})(y) = c u(y) + c
H_\alpha(u^-|B_{1})(y)
\,,
\end{align*}
which proves the second inequality in \eqref{eq:harnack_main_stable}. The first
one is proved analogously.

Inequality \eqref{eq:harnack_alpha_pos} is proved as follows. Assume $u$ is
nonnegative in $B_1$. Using the same strategy as above we obtain for some
$c_1, c_2 >0$ and $c=\max(c_1, c_2)$
\begin{align*}
 u(x) &\leq c_1 H_\alpha(u|B_{\frac34})(y) + c_1 H_\alpha(u^-|B_{\frac34})(y) \\
&\leq c_1 u(y) + c_2  \alpha(2{-}\alpha) \il_{\R^d \setminus B_1}
\frac{u^-(z)}{(|z|^2-(\frac34)^2)^{\alpha/2} |z-y|^{d}} \, dz \\
&\leq c u(y) + c  \alpha(2{-}\alpha) \il_{\R^d \setminus B_1}
\frac{u^-(z)}{|z|^{d+\alpha}} \, dz   \,. 
\end{align*}
The proof of the theorem is complete. Note that different versions of this
result have been announced in \cite{Kas11}.
\end{proof}

\pagebreak[3]
Let us make some observations:
\begin{enumerate}
\item There is no assumption on the sign of $u$ needed
for \eqref{eq:harnack_main_stable}. Inequality \eqref{eq:harnack_main_stable}
does hold in the classical case $\alpha=2$, too.
\item If $u$ is nonnegative in all of $\R^d$ ($\alpha \in (0,2)$) or
nonnegative in $B_1$ ($\alpha=2$), then the second inequality in 
\eqref{eq:harnack_main_stable} reduces to the well-known formulation of the 
Harnack inequality.
\item If $u$ is nonnegative in $B_1$, then \eqref{eq:harnack_alpha_pos} reduces
for $\alpha \to 2$ to the original Harnack inequality.
\item For the above results, one might want to impose regularity conditions on
$u$ such that
$\Delta^{\alpha/2} u(x)$ exists at every point $x \in B_1$, e.g. $u|_{B_1} \in
C^2(B_1)$ and $u(x)/(1+|x|^{d+\alpha}) \in L^1(\R^d)$. However, the assumption
that the integral in \eqref{eq:def_frac-Laplace} converges, is sufficient.
\end{enumerate}

The proof of \autoref{theo:harnack_frac_new} does not use the
special structure of $\Delta^{\alpha/2}$. The proof only uses the
decomposition $u=u^+ - u^-$ and the Harnack inequality
for
the Poisson kernel. Roughly speaking, it holds for every
linear operator that satisfies a maximum principle. One more abstract way of
formulating this result in a general framework is as follows:

\begin{lemma}\label{lem:trivial} Let $(X,\cW)$ be a balayage space (see
\cite{BlHa86}) such that $1\in \cW$. 
Let $V,W$ be open sets in $X$ with $\ov V \subset W$. Let $c>0$. Suppose that,
for all $x,y \in V$ and $h \in \cH_b^+(V)$,
\begin{align}\label{eq:harnack+}
u(x)\le c u(y).
\end{align}
Then $\epsx^{V^c} \le c\epsy^{V^c}$ and, for every $u\in \cH_b(W)$,
\begin{align}\label{eq:harnack}
u(x)\le c u(y)+c\int u^- \,d\eps_{y}^{V^c}\,.
\end{align}
\end{lemma}
Here, $\cH_b(A)$ denotes the set of bounded functions which are harmonic in the
Borel set $A$. Functions in $\cH_b^+(A)$, 
in addition, are nonnegative.

\begin{proof}
Since, for every positive continuous function $f$ with compact support the
mapping $f \mapsto  \eps_z^{V^c}(f)$ belongs to $\cH_b^+(V)$, the first
statement follows. Let $u\in \cH_b(W)$. Then $u(x)=\epsx^{V^c}(u)$,
$u(y)=\epsy^{V^c}(u)$ and hence
\begin{align*}
u(x)\le \epsx^{V^c}(u^+)\le c\epsy^{V^c}(u^+)
               = c\epsy^{V^c} (u+u^-) = cu(y)+c \il u^{-} \, d\eps_{y}^{V^c}.
\end{align*}
\end{proof}

\section{Functional inequalities and scaling property}\label{sec:prelims}

In this section we collect several auxiliary results. In particular, we will 
need some properties of the Sobolev spaces
$H^{\alpha/2}(D)$. The following fact about extensions has
an elementary proof, see \cite{DiNezzaPalatucciValdinoci}. However, one has to
go through it and see that the constants do not depend on $\alpha$, provided one
has the factor $(2-\alpha)$ in front of the Gagliardo norm,
cf. \eqref{eq:def_mu-alpha} and \eqref{eq:Sobolevnorm}.

\begin{fact}[Extension] \label{fact:extension}
Let $D\subset \R^d$ be a bounded Lipschitz domain, and let $0<\alpha<2$.
Then there exists a constant $c=c(d, D)$, which is independent of $\alpha$, and
an extension operator $E:H^{\alpha/2}(D) \to H^{\alpha/2}(\R^d)$ with
norm $\|E\| \leq c$.
\end{fact}

Furthermore, we will need the following Poincar\'e inequality, cf.
\cite{Ponce04}.
\begin{fact}[Poincar\'{e} I]\label{fact:poincare-i}
Let $D\subset \R^d$ be a bounded Lipschitz domain, and let
$0 < \alpha_0 \leq \alpha<2$.
Then there exists a constant $c=c(d, \alpha_0, D)$, which is independent of
$\alpha$, such that
\begin{align}\label{eq:poincare-i}
\|u - \frac{1}{|D|}\int_D u\, \d x \|_{L^2(D)}^2 \leq c \cE^{\mu_\alpha}_D(u,u) 
\qquad
(u\in H^{\alpha/2}(D)).
\end{align}
\end{fact}

The following results, \autoref{fact:poincare-ii} and \autoref{fact:sobolev},
are standard for fixed $\alpha$. For $\alpha \to 2$ they follow from results in 
\cite{BBM01},\cite{MaSh02}, \cite{Ponce04}. They are established in the case
when $B_r(x)$ denotes the cube of all $y \in \R^d$ such that
$|y_i -x_i| < r$ for any $i \in \{1,\ldots, d\}$. They hold true for balls
likewise.

\begin{fact}[Poincar\'{e} II]\label{fact:poincare-ii} Assume $\alpha_0,
\eps >0$ and $0 < \alpha_0 \leq \alpha<2$. There exists a constant $c$, which is
independent of
$\alpha$, such that for $B_R=B_R(x_0)$
\begin{align*}
u \in H^{\alpha/2}(B_R), \qquad \left|B_R \cap \{u = 0\}\right| \geq \eps |B_R|
\end{align*}
implies
\begin{align}\label{eq:poincare-ii}
\il_{B_R} \big(u(x)\big)^2 \, \d x \leq c R^\alpha \iil_{B_R B_R}
\frac{\big(u(y)-u(x)\big)^2}{|x-y|^{d+\alpha}} \,\d y \, \d x \,.
\end{align}
\end{fact}

\begin{fact}[Sobolev embedding]\label{fact:sobolev} Assume $d\in \N, d \geq 2,
R_0>0$, and $0<\alpha_0 \leq \alpha<2$, $q \in [1,\frac{2d}{d-\alpha}]$. Then
there exists a constant $c$, which is independent of $\alpha$, such that for
$R\in(0,R_0)$ and $u \in H^{\alpha/2}(B_R)$
\begin{align*}
\Big(\il_{B_R} |u(x)|^\frac{2d}{d-\alpha} \, \d x \Big)^\frac{d-\alpha}{d}  \leq
c
\iil_{B_R B_R} \frac{\big(u(y)-u(x)\big)^2}{|x-y|^{d+\alpha}} \,\d y \, \d x 
+ c
R^{-\alpha + \frac{d(q-2)}{q}} \Big(\il_{B_R} |u(x)|^q \, \d x \Big)^\frac{2}{q}
\,.
\end{align*}
\end{fact}

When studying nonlocal bilinear forms on bounded sets, it is natural to work
with function spaces which impose some regularity of the functions across the
boundary. These spaces seem not be part of the standard literature which is why
we provide a small introduction.

We often make use of scaling and translations. Our main assumptions, conditions 
\eqref{eq:assum_comp} and \eqref{eq:assum_cutoff} assure a certain behavior 
of the family of measures $\mu$ with respect to the unit ball $B_1 \subset 
\R^d$. Let us formulate these conditions with respect to general balls 
$B_r (\xi) \subset \R^d$.

Given $\xi \in \R^d, r > 0, A \geq
1$, we say that $\mu$ satisfies \eqref{eq:assum_comp-scaled} if:
\begin{align} 
\label{eq:assum_comp-scaled}\tag{A;$\xi$,r} 
\begin{split}
& \text{ For every ball } B_{\rho}(x_0) \text{ with } \rho \in (0,r), x_0 \in 
B_{r}(\xi) \text{ and every } v \in H^{\alpha/2}(B_\rho(x_0)): \\
& 
\qquad\qquad A^{-1} \, \cE^\mu_{B_\rho(x_0)}(v,v) \leq 
\cE^{\mu_\alpha}_{B_\rho(x_0)}(v,v)
\leq A
\,
\cE^\mu_{B_\rho(x_0)}(v,v) \,. 
\end{split}
\end{align}
Given $\xi \in \R^d, r > 0, B \geq
1$, we say that $\mu$ satisfies \eqref{eq:assum_cutoff-scaled} if:
\begin{align} \label{eq:assum_cutoff-scaled}\tag{B;$\xi$,r}
\begin{split}
&\text{ For } 0 < \rho \leq R \leq r \text{ and } x_0 \in B_r(\xi) \text{ there 
is a nonnegative measurable
function }
\\ 
& \tau : \R^d \to \R \text{ with }
\operatorname{supp}(\tau) \subset \overline{B_{R+\rho}(x_0)}, \tau(x) \equiv
1 \text{ on } B_{R}(x_0)\,, \|\tau\|_\infty \leq 1\,, \text{ and } \\
& \qquad \sup\limits_{x \in \R^d} \; \il_{\R^d} \big(\tau(y) - \tau(x)\big)^2  
\mu(x,\d y) \leq B \, \rho^{-\alpha} \,.
\end{split}
\end{align}

Let us explain how the operator under
consideration behaves with respect to rescaled functions.

\begin{lemma}[Scaling property] \label{lem:scaling} 
Assume $\xi \in \R^d$ and $r\in 
(0,1)$. Let $u \in V^\mu_{B_{r}(\xi)}(\R^d)$ 
satisfy
$\cE^\mu(u,\phi) \geq (f,\phi)$ for every nonnegative $\phi \in
H^\mu_{B_{r}(\xi)}(\R^d)$. Define a
diffeomorphism
$J$ by $J(x)= rx + \xi$.   Define rescaled versions $\widetilde f$, $\widetilde
u$ of $u$ and $f$ by $\widetilde u(x)= u(J(x))$ and $\widetilde f$ by
$\widetilde f(x)=r^\alpha
f(J(x))$.
\begin{enumerate}
\item Then $\widetilde u$ satisfies for all nonnegative $\phi \in 
H^{\widetilde \mu}_{B_{1}}(\R^d)
$
\begin{align*}
\cE^{\widetilde{\mu}}(\widetilde u,\phi) = \iil_{\R^d \R^d}
\big(\widetilde{u}(y)-\widetilde{u}(x)\big)\big(\phi(y)-\phi(x)\big)
\widetilde{\mu}(x,\d
y) \, \d x \geq (\widetilde{f},\phi)\,,
\end{align*}
where 
\begin{align}
\widetilde{\mu} (x,\d y) = r^\alpha \mu_{J^{-1}}(J(x), \d y) \text{ and
} \mu_{J^{-1}}(z, A) = \mu(z, J(A)) \,.
\end{align}
\item Assume $\mu$ satisfies conditions \eqref{eq:assum_comp-scaled},
\eqref{eq:assum_cutoff-scaled} for some $\alpha \in
(0,2)$ and $A \geq 1$, $B \geq 1$, $\xi \in \R^d$, $r > 0$. Then the family of 
measures
$\widetilde{\mu}=\widetilde{\mu} (\cdot,\d y)$ satisfies assumptions
\eqref{eq:assum_comp} and \eqref{eq:assum_cutoff} with the same constants.
\end{enumerate}
\end{lemma}

\begin{remark}
The condition \eqref{eq:assum_largejumps} is affected by scaling in a
non-critical way. We deal with this phenomenon further below
in \autoref{sec:weak_solutions} and \autoref{sec:regularity}
\end{remark}

\begin{proof} For the proof of the first statement, let $\phi \in H^{\widetilde 
\mu}_{B_{1}}(\R^d)$ be a nonnegative test function. Define $\phi_r \in
H^{\mu}_{B_{r}(\xi)}(\R^d) $ by $\phi_r = \phi \circ 
J^{-1}$. Then
\begin{align*}
 \iint &\left(\widetilde u(y) - \widetilde u(x)\right) \left(\phi(y) -
\phi(x)\right) \widetilde \mu(x,\d y) \d x \\
&= r^\alpha \iint \left(u(J(y)) - u(J(x))\right) \left(\phi_r(J(y)) -
\phi_r(J(x))\right) \mu_{J^{-1}}(J(x),\d y) \d x \\
&= r^{\alpha-d} \iint \left(u(J(y)) - u(x)\right) \left(\phi_r(J(y)) -
\phi_r(x)\right) \mu_{J^{-1}}(x,\d y) \d x \\
&= r^{\alpha-d} \iint \left(u(y) - u(x)\right) \left(\phi_r(y) -
\phi_r(x)\right) \mu(x,\d y) \d x \\
&\geq r^{\alpha-d} \int f(x) \phi_r(x) \d x = \int r^{\alpha} f(J(x)) \phi(x) \d
x = \int \widetilde f(x) \phi(x) \d x \,,
\end{align*}
which is what we wanted to prove. Let us now prove that $\widetilde \mu$
inherits properties \eqref{eq:assum_comp}, \eqref{eq:assum_cutoff} from $\mu$
with the same constants $A$ and $B$. Let us only consider 
the case $\xi=0$. 
In order to verify condition \eqref{eq:assum_comp} we need to consider an 
arbitrary ball $B_\rho(x_0)$ with $\rho \in (0,1)$ and $x_0 \in B_1$. Let us 
simplify the situation further by assuming $x_0 = 0$. The general case can be 
proved analogously. Thus, we assume $r \in
(0,1)$ and $u \in H^{\alpha/2}(B_\rho)$. The estimate $\cE^{\widetilde 
\mu}_{B_\rho}(u,u) \leq
A \cE^{\mu_\alpha}_{B_\rho}(u,u)$ can be derived as follows. Define a function 
$\widehat
u \in H^{\alpha/2}(B_{r \rho})$ by $\widehat u = u \circ J^{-1}$. Then
{\allowdisplaybreaks
\begin{align*}
 \cE^{\widetilde \mu}_{B_\rho}(u,u) &= \int\limits_{B_\rho} 
\int\limits_{B_\rho} \left(
u(y) - u(x) \right)^2 \widetilde \mu(x, \d y) \d x = r^\alpha 
\int\limits_{B_\rho}
\int\limits_{B_\rho} \left(
\widehat u(J(y)) - \widehat u(J(x)) \right)^2 \mu_{J^{-1}}(J(x), \d y) \d x \\
&= r^{\alpha-d} \int\limits_{B_{r\rho}}
\int\limits_{B_r} \left(
\widehat u(J(y)) - \widehat u(x) \right)^2 \mu_{J^{-1}}(x, \d y) \d x \\
&= r^{\alpha-d} \int\limits_{B_{r \rho}}
\int\limits_{B_{r \rho}} \left(
\widehat u(y) - \widehat u(x) \right)^2 \mu(x, \d y) \d x \leq A r^{\alpha-d}
\int\limits_{B_{r \rho}}
\int\limits_{B_{r \rho}} \frac{\left(
\widehat u(y) - \widehat u(x) \right)^2}{|x-y|^{d+\alpha}} \d y \d x \\
&= A r^{-2d}
\int\limits_{B_{r \rho}} \int\limits_{B_{r \rho}} \frac{\left(
u(J^{-1}(y)) - u(J^{-1}(x)) \right)^2}{|J^{-1}(x)-J^{-1}(y)|^{d+\alpha}} \d y
\d x = A \int\limits_{B_{\rho}} \int\limits_{B_{\rho}} \frac{\left(
u(y) - u(x) \right)^2}{|x-y|^{d+\alpha}} \d y \d x \,, 
\end{align*}
}
which proves our claim. The estimate $\cE^{\mu_\alpha}_{B_\rho}(u,u) \leq A \,
\cE^{\widetilde \mu}_{B_\rho}(u,u)$ follows in the same way.

In order to check condition \eqref{eq:assum_cutoff} for $\widetilde \mu$ we
proceed as follows. Again, we assume $x_0 = 0$, $r \in (0,1)$. The general case 
can be proved analogously. Assume 
$R,\rho
\in (0,1)$. Let $\widehat \tau: \R^d \to \R $ satisfy
$\operatorname{supp}(\widehat \tau) \subset \overline{B_{rR+r\rho}}, \,
\widehat \tau \equiv 1 \text{ on } B_{rR}$ and 
\begin{align*} \sup\limits_{x \in \R^d} & \; \il_{\R^d} \big(\widehat
\tau(y) - \widehat \tau(x)\big)^2 \mu(x,\d y) \leq B (r\rho)^{-\alpha} \\
&\Leftrightarrow  \quad \sup\limits_{x \in \R^d} \; \il_{\R^d} \big(\widehat
\tau(y) - \widehat \tau(J(x))\big)^2 \mu(J(x),\d y) \leq B (r\rho)^{-\alpha} \,.
\end{align*}
Such a function $\widehat \tau$ exists because, by assumption, $\mu$ satisfies
\eqref{eq:assum_cutoff-scaled}. Next, define $\tau = \widehat \tau \circ J$. 
Then
$\tau$ satisfies $\operatorname{supp}(\tau) \subset
\overline{B_{R+\rho}}, \, \tau \equiv 1
\text{ on } B_{R}$ and, by a change of variables, 
\begin{align*} \sup\limits_{x \in \R^d}  \; \il_{\R^d} \big(\tau(y) -
\tau(x)\big)^2 \widetilde \mu(x,\d y) &= r^\alpha \sup\limits_{x \in \R^d} \;
\il_{\R^d} \big(\widehat \tau(J(y)) - \widehat \tau(J(x))\big)^2 \mu_{J^{-1}}
(J(x),\d y) \\
&= r^\alpha \sup\limits_{x \in \R^d} \;
\il_{\R^d} \big(\widehat \tau(y) - \widehat \tau(J(x))\big)^2 \mu
(J(x),\d y) \quad \leq B \rho^{-\alpha} \,,
\end{align*}
which shows that $\widetilde \mu$ satisfies \eqref{eq:assum_cutoff} with the
constant $B$. The proof of the lemma is complete.
\end{proof}

\section{The weak Harnack inequality for nonlocal 
equations}\label{sec:weak_solutions}

The main aim of this section is to provide a proof of the weak Harnack
inequality \autoref{theo:weak_harnack_nonlocal}. The key result of this 
section is the corresponding result for supersolutions that are nonnegative in 
all of $\R^d$:

\begin{theorem} \label{theo:weak_harnack_f}
Assume $f \in L^{q/\alpha}(B_1)$ for some $q > d$, $\alpha \in
[\alpha_0,2)$. There are positive reals $p_0, c$ such that for every $u \in
V_{B_{1}}^\mu(\R^d) $ with $u \geq 0$ in $\R^d$ satisfying 
\begin{align*} \cE(u,\phi) \geq (f,\phi) \text{ for every nonnegative } \phi \in
H^\mu_{B_1}(\R^d) \end{align*}
the following holds:
\[   \inf\limits_{B_{\frac14}} u \geq c \big( \fint\limits_{B_\frac{1}{2}}
u(x)^{p_0} \,
\d x \big)^{1/p_0} - \|f\|_{L^{q/\alpha}(B_{\frac{15}{16}})} \,. \] 
The constants $p_0, c$ depend only on  $d, \alpha_0, A, B$. They are 
independent of 
$\alpha \in [\alpha_0,2)$.
\end{theorem}

\begin{remark}
All results in this section are robust with respect to $\alpha \in
[\alpha_0,2)$, i.e. constants do not depend on $\alpha$. 
\end{remark}

The main application of this result is the proof of 
\autoref{theo:weak_harnack_nonlocal}.

\begin{proof} Set $u=u^+ - u^-$. The assumptions imply for any nonnegative $\phi
\in H^\mu_{B_1}(\R^d)$ 
\begin{align*}
 \cE(u^+,\phi) \geq \cE(u^-,\phi) + (f, \phi) = \il_{B_1} \phi(x) \big( f(x) - 
2 \il_{\R^d \setminus B_1} u^-(y) \mu(x, \d y) \big) \, \d x \,, 
\end{align*}
i.e. $u^+$ satisfies all assumptions of \autoref{theo:weak_harnack_f} with $q=+ 
\infty$ and $\wt{f}:B_1 \to \R$ defined by
\[ \wt{f}(x) = f(x) - 2 \il_{\R^d \setminus B_1} u^-(y) \mu(x, \d y) \,.\]
The assertion of the theorem is true if $\supl_{x \in B_{\frac{15}{16}}} 
\il_{\R^d \setminus B_1} u^-(y) \mu(x, \d y)$ is infinite. Thus we can assume 
this quantity to be finite. \autoref{theo:weak_harnack_f} now implies
\begin{align*}   
\inf\limits_{B_{\frac14}} u \geq c_1 \big( \fint\limits_{B_\frac{1}{2}}
u(x)^{p_0}
\, \d x
\big)^{1/p_0} - c_2 \supl_{x \in B_{\frac{15}{16}}} \big(\il_{\R^d \setminus 
B_1} 
u^-(y)
 \mu(x, \d y) \big) - \|f\|_{L^{q/\alpha}(B_{\frac{15}{16}})} 
\end{align*} 
for some positive constants $c_1, c_2$. The proof is complete. 
\end{proof}

By scaling and translation, we obtain the following corollary.

\begin{corollary} \label{cor:weak_harnack}
Let $x_0 \in \R^d$, $R \in
(0,1)$. Assume $\mu$ is a family of measures satisfying 
\eqref{eq:assum_comp-scaled} and 
\eqref{eq:assum_cutoff-scaled}. Assume $u \in V^{\mu}_{B_R(x_0)}(\R^d)$ 
satisfies $u \geq 0$ in $B_R(x_0)$ and $\cE(u,\phi) \geq 0$ for every 
nonnegative $\phi \in H^{\mu}_{B_R(x_0)}(\R^d)$. Then 
\[   \inf\limits_{B_{\frac{R}{4}}(x_0)} u \geq c \big(
\fint\limits_{B_{\frac{R}{2}}(x_0)}
u(x)^{p_0} \, \d x \big)^{1/p_0} - R^\alpha \sup\limits_{x \in
B_{\frac{15 R}{16}}(x_0)}
\int\limits_{\R^d \setminus B_R(x_0)} u^-(y) \mu(x , \d y) \,, \] 
with positive constants $p_0, c$ which depend only on  $d, \alpha_0, A, 
B$. In particular, they are independent of $\alpha \in [\alpha_0,2)$.
\end{corollary}

Let us proceed to the proof of \autoref{theo:weak_harnack_f}. 

\begin{remark}
Without further mentioning we assume that $\mu$ is a family of measures that 
satisfies \eqref{eq:assum_comp} and \eqref{eq:assum_cutoff} for some $A \geq 1, 
B \geq 1$ and $\alpha_0 \leq \alpha < 2$. The constants in the assertions below 
depend, among other things, on  $A, B$, and $\alpha_0$. They do not depend on 
$\alpha$, though.
\end{remark}

Let us first establish several auxiliary results. Our approach is closely 
related to the approach in \cite{Kas09} from where we borrow the following 
technical lemma, cf. \cite[Lemma 2.5]{Kas09}.

\begin{lemma}\label{lem:thewonder}
Let $a,b >0$, $p>1$ and $\tau_1, \tau_2 \geq 0$. Then
\begin{align}\label{eq:thewonder}
\begin{split}
(b-a) &\big( \tau_1^{p+1} a^{-p} - \tau_2^{p+1}
b^{-p} \big) \\
& \geq  \tfrac{\tau_1 \tau_2}{p-1}    \big( (\tfrac{b}{\tau_2})^\frac{-p+1}{2} -
(\tfrac{a}{\tau_1})^\frac{-p+1}{2} \big)^2 - \max\{4, \tfrac{6p-5}{2}\} (\tau_2
- \tau_1)^2 \big(
(\tfrac{b}{\tau_2})^{-p+1} +
(\tfrac{a}{\tau_1})^{-p+1} \big)\,.
\end{split}
\end{align}
\end{lemma}

The next result is an extension of corresponding results in \cite{Kas09} and 
\cite{BBCK09}.

\begin{lemma} \label{lem:morrey-log}
Assume $0 < \rho < r < 1$ and $z_0 \in B_1$. Set $B_{r}=B_{r}(z_0)$. 
Assume $f \in L^{q/\alpha}(B_{2r})$ for some $q > d$. Assume $u \in
V^{\mu}_{B_{2r}}(\R^d)$ is nonnegative in $\R^d$ and satisfies
\begin{align*}
\cE(u,\phi) &\geq (f,\phi) \text{ for any nonnegative } \phi \in
H^{\mu}_{B_{2r}}(\R^d)  \\
u(x) &\geq \eps \quad \text{ for almost all } x \in B_{2r} \text{ and some }
\eps > 0 \,. \end{align*}
Then 
\begin{align}
\iil_{B_{r} B_{r}} & \Big( \suml^\infty_{k=1} \frac{ \left( \log u(y) - \log
u(x) \right)^{2k}}{(2k)!} \Big) \mu(x, \d y) \d x \\ & \leq c 
\rho^{-\alpha}
|B_{r+\rho}| 
+ \eps^{-1} \|f\|_{L^{q/\alpha}(B_{r+\rho})}
\|\mathbbm{1}\|_{L^{q/(q-\alpha)}(B_{r+\rho})}\,, \label{eq:morrey-log-ext}
\end{align}
where $c>0$ is independent of $u, x_0, r, \rho,f, \eps, \alpha$. 
\end{lemma}
Note that for $\eps \geq c_1 (r+\rho)^\delta \|f\|_{L^{q/\alpha}(B_{r+\rho})}$
with $\delta=\alpha (\frac{q-d}{q})$ one obtains 
\begin{align}
\iil_{B_{r} B_{r}} \Big( \suml^\infty_{k=1} \frac{ \left( \log u(y) - \log u(x)
\right)^{2k}}{(2k)!} \Big) \mu(x, \d y) \d x \leq c_2 
\rho^{-\alpha}
|B_{r+\rho}|
\,.
\end{align}

From the above lemma it will be deduced that $\log {u} \in$ BMO $(B_1)$ where 
BMO $(B_1)$ contains all functions of bounded mean oscillations \cite{JoNi61}. 

\begin{proof} The proof uses several ideas developed in \cite{BBCK09}. Let 
$\tau :\R^d \to \R$ be a function according to \eqref{eq:assum_cutoff}, i.e. 
more precisely we assume 
\begin{align*}
\begin{cases}
\opn{supp}(\tau) \subset \overline{B_{r+\rho}} \subset B_{2r},  \|\tau\|_\infty
\leq 1 \,, \tau \equiv 1 \text{ on } B_r, \\
\sup\limits_{x \in \R^d} \; \il_{\R^d} \big(\tau(y) - \tau(x)\big)^2  
\mu(x,\d y) \leq B \rho^{-\alpha} \,. 
\end{cases}
\end{align*}

Then    
\begin{align}\label{eq:gradtauest}
\begin{split}
\iil_{\R^d \R^d} &  \big(\tau(y) - \tau(x) \big)^2 \mu(x, \d y) 
\d x \\
& = \iil_{B_{r+\rho} B_{r+\rho}}  \big(\tau(y) - \tau(x) \big)^2 \mu(x, \d y) 
\d x + 2 \iil_{B_{r+\rho} B_{r+\rho}^c} \big(\tau(y) - \tau(x) \big)^2 
\mu(x, \d y) 
\d x \\
& \leq 2 \iil_{B_{r+\rho} \R^d} \big(\tau(y) - \tau(x) \big)^2 
\mu(x, \d y) \d x \\
& \leq  2 |B_{r+\rho}| \sup\limits_{x \in \R^d} \; \il_{\R^d} \big(\tau(y) - 
\tau(x)\big)^2 \mu(x,\d y) \\
& \leq 2 c \rho^{-\alpha} |B_{r+\rho}| \,.
\end{split}
\end{align}

We choose $\phi(x) = - \tau^2(x) u^{-1}(x)$ as a test function. Denote
$B_{r+\rho}$ by $B$. We obtain
{\allowdisplaybreaks
\begin{align}\label{eq:split1}
\begin{split}
(f,\phi) &\geq  \iil_{\R^d \R^d} \big(u(y) - u(x)\big) \big(\tau^2(x) u^{-1}(x)
- \tau^2(y) u^{-1}(y)\big) \; \mu(x, \d y) \, \d x \\
    &= \iil_{B B} \tau(x)\tau(y) \Big( \frac{\tau(x)u(y)}{\tau(y)u(x)} +
\frac{\tau(y) u(x)}{\tau(x) u(y)} - \frac{\tau(y)}{\tau(x)}
-\frac{\tau(x)}{\tau(y)} \Big) \, \mu(x, \d y) \, \d x \\
&\quad + 2 \iil_{B B^c} \big(u(y) - u(x)\big) \big(\tau^2(x) u^{-1}(x) -
\tau^2(y) u^{-1}(y)\big) \; \mu(x, \d y) \, \d x \\
&\quad + \iil_{B^c B^c} \big(u(y) - u(x)\big) \big(\tau^2(x) u^{-1}(x) -
\tau^2(y) u^{-1}(y)\big) \; \mu(x, \d y) \d x  \,.
\end{split}
\end{align}
} 
Setting $A(x,y)=\frac{u(y)}{u(x)}$ and $B(x,y)=\frac{\tau(y)}{\tau(x)}$ we
obtain
{\allowdisplaybreaks
\begin{align*}
\iil_{B B} &\tau(x)\tau(y) \Big( \frac{A(x,y)}{B(x,y)}  + \frac{B(x,y)}{A(x,y)}
- B(x,y) - \frac{1}{B(x,y)} \Big) \, \mu(x, \d y) \, \d x \\
& = \iil_{B B} \tau(x)\tau(y) \left[ \Big(\frac{A(x,y)}{B(x,y)} +
\frac{B(x,y)}{A(x,y)} -2 \Big)  - \Big(\sqrt{B(x,y)}-\frac{1}{\sqrt{B(x,y)}}
\Big)^2 \right]  \; \mu(x, \d y) \, \d x
\\
& = \iil_{B B} \tau(x)\tau(y) \Big( 2 \suml^\infty_{k=1} \frac{ \left(\log
A(x,y) - \log B(x,y)\right)^{2k}}{(2k)!} \Big) \, \mu(x, \d y) \d x  \\
& \qquad - \iil_{B B} \tau(x)\tau(y)
\Big(\sqrt{B(x,y)}-\frac{1}{\sqrt{B(x,y)}}\Big)^2  \, \mu(x, \d y) \d x \\
& = \iil_{B B} \tau(x)\tau(y) \Big( 2 \suml^\infty_{k=1} \frac{ \left( \log
\frac{u(y)}{\tau(y)} - \log \frac{u(x)}{\tau(x)} \right)^{2k}}{(2k)!} \Big) \, 
\mu(x, \d y) \d x \\
& \qquad -  \iil_{B B} \big( \tau(x) - \tau(y) \big)^2 \, \mu(x, \d y) \, \d
x \\
& \geq \il_{B_r} \il_{B_r} \Big( 2 \suml^\infty_{k=1} \frac{ \left( \log u(y) -
\log u(x) \right)^{2k}}{(2k)!} \Big) \, \mu(x, \d y) \, \d x -  
\iil_{B B}
\big(
\tau(x) - \tau(y) \big)^2 \, \mu(x, \d y) \, \d x \,,
\end{align*}
}where we applied \eqref{eq:gradtauest} and the fact that for positive real
$a,b$
\begin{align}
\frac{(a-b)^2}{ab} = (a-b) (b^{-1} - a^{-1}) = (\log a - \log
b)^2 + 2 \suml^\infty_{k=2} \frac{(\log a - \log b)^{2k}}{(2k)!}  \;.
\label{Adam}
\end{align}
Altogether, we obtain 
\begin{align}\label{eq:split2}
\begin{split}
(f,\phi) &\geq \il_{B_r} \il_{B_r} \Big( 2 \suml^\infty_{k=1} \frac{ \left( \log
u(y) - \log u(x) \right)^{2k}}{(2k)!} \Big) \, \mu(x, \d y) \d x -  \iil_{B
B}
\big( \tau(x) - \tau(y) \big)^2 \, \mu(x, \d y) \d x \\
&\quad + 2 \iil_{B_{r+\rho} B_{r+\rho}^c} \big(u(y) - u(x)\big) \big(\tau^2(x)
u^{-1}(x) - \tau^2(y) u^{-1}(y)\big) \; \mu(x, \d y) \d x \,.
\end{split}
\end{align}
The third term on the right-hand side can be estimated as follows:

\begin{align*}
2 &\iil_{B_{r+\rho} B_{r+\rho}^c} \big(u(y) - u(x)\big) \big(\tau^2(x) u^{-1}(x)
- \tau^2(y) u^{-1}(y)\big) \; \mu(x, \d y) \d y \\
&= 2 \iil_{B_{r+\rho} B_{r+\rho}^c} \big(u(y) - u(x)\big) \big(- \tau^2(y)
u^{-1}(y)\big) \; \mu(x, \d y) \d y \\
&= 2 \il_{B_{r+\rho}} \il_{B_{r+\rho}^c} \tfrac{\tau^2(y)}{u(y)}  u(x) \,
\mu(x,\d y) \d x - 2 \il_{B_{r+\rho}} \il_{B_{r+\rho}^c} 
\tau^2(y) \, \mu(x, \d y) \, \d x \\ 
&\geq - 2 \ilRd\ilRd \big(\tau(y)-\tau(x)\big)^2 \mu(x, \d y) \d x \,,
\end{align*}

where we used nonnegativity of $u$ in $\R^d$. Therefore, 
\begin{align}\label{eq:split3}
\begin{split}
\il_{B_r} &\il_{B_r} \Big( 2 \suml^\infty_{k=1} \frac{ \left( \log u(y) - \log
u(x) \right)^{2k}}{(2k)!} \Big) \, \mu(x, \d y)  \d x \\
&\leq 2 \iil_{\R^d \R^d} \big( \tau(x) - \tau(y) \big)^2 \, \mu(x, \d y) \d x
+ \|f\|_{L^{q/\alpha}(B_{r+\rho})} \|u^{-1}\|_{L^{q/(q-\alpha)}(B_{r+\rho})} \,.
\end{split}
\end{align}
The proof is complete after the trivial observation $|u^{-1}| \leq
\eps^{-1}$.
\end{proof}

\begin{lemma} \label{lem:moser-flip} Assume $0<R<1$ and $f \in
L^{q/\alpha}(B_{\frac{5R}{4}})$ for some $q > d$. Assume $u \in 
V^\mu_{B_{\frac{5R}{4}}}(\R^d)$ is nonnegative in $\R^d$ and satisfies
\begin{align*}
\cE(u,\phi) &\geq (f,\phi) \text{ for any nonnegative } \phi \in
H^\mu_{B_{\frac{5R}{4}}}(\R^d) \,,\\
u(x) &\geq \eps \quad \text{ for almost all } x \in B_{\frac{5R}{4}} \text{ and
some }
\eps > \tfrac14 R^\delta \|f\|_{L^{q/\alpha}(B_{\frac{9R}{8}})} \,,
\end{align*}
where $\delta=\alpha (\frac{q-d}{q})$. Then there exist $\ov{p} \in (0,1)$ and
$c>0$ such that ,
\begin{align}
\left(\fint_{B_R} u(x)^{\ov{p}} \; \d x \right)^{1/\ov{p}} \; \d x \leq c
\left(\fint_{B_R} u(x)^{-\ov{p}} \; \d x \right)^{-1/\ov{p}}\,,
\label{eq:flip-assert}
\end{align}
where $c$ and $\ov{p}$ are independent of $x_0, R, u$, $\eps$, and $\alpha$.
\end{lemma}

\begin{proof}  
The main idea is to prove $\log u \in$ BMO$(B_R)$. Choose $z_0 \in B_R$ and $r >
0$ such that $B_{2r}(z_0) \subset B_{\frac{R}{8}}$. 
Set $\rho=r$. \autoref{lem:morrey-log} and Assumption
\eqref{eq:assum_comp} imply
\begin{align*}
& \il_{B_r(z_0)} \il_{B_r(z_0)} \frac{\big(\log u(y) - \log u(x)
\big)^2}{|x-y|^{d+\alpha}} \; \d y \, \d x \\
\leq  &\il_{B_r(z_0)} \il_{B_r(z_0)} \big(\log u(y) - \log u(x) \big)^2 
\mu(x, \d y) \d x \leq  c_1 r^{d-\alpha} \,.
\end{align*}
Application of the Poincar\'{e} inequality, 
\autoref{fact:poincare-i}, and the scaling property \eqref{fact:poincare-ii} 
leads to 
\begin{align}
\il_{B_r(z_0)} \left|\log u(x) - [\log u]_{B_r(z_0)}\right|^2 \; \d x \leq  c_2
r^{d}  \,,
\end{align}
where $[\log u]_{B_r(z_0)} = |B_r(z_0)|^{-1}\il_{B_r(z_0)} \log
u=\fint_{B_r(z_0)} \log u$. From here 
\begin{align*}
\il_{B_r(z_0)} \left|\log u(x) - [\log u]_{B_r(z_0)}\right| \, \d x \leq 
\Big(\il_{B_r(z_0)} \left|\log u(x) - [\log u]_{B_r(z_0)}\right|^2 \;
\d x\Big)^{\frac 12} |B_r(z_0)|^{\frac 12} \leq c_3 r^{d}  \,.
\end{align*}

An application of the John-Nirenberg embedding, see \cite[Chapter 7.8]{GiTr83}, 
then gives
\begin{align*}
\il_{B_R} e^{\ov{p} \left|\log u(y)-[\log u]_{B_r}\right|} \; \d y \leq c_4 R^d
\,,
\end{align*}
where $\ov{p}$ and $c_4$ depend only on $d$ and $c_3$. One obtains
\begin{align*}
&\Big( \il_{B_R} u(y)^{\ov{p}} \; \d y \Big) \Big( \il_{B_R} u(y)^{-\ov{p}} \;
\d y
\Big) \\
& \quad = \il_{B_R} e^{\ov{p} (\log u(y)-[\log u]_{B_r}) } \; \d y  \times
\il_{B_R} e^{- \ov{p} (\log u(y)-[\log u]_{B_r}) } \; \d y \leq c_4^2 R^{2d} \,.
\end{align*}
The above inequality proves assertion
\eqref{eq:flip-assert}. \autoref{lem:moser-flip} is proved.
\end{proof}

The next result allows us to apply Moser's iteration for negative exponents. It
is a purely local result although the Dirichlet form is nonlocal. 

\begin{lemma} \label{lem:moser-neg-it} 
Assume $x_0 \in B_1$ and $0 < 4 \rho < R < 1-\rho$. Set $B_{R}=B(x_0,R)$. 
Assume $f \in L^{q/\alpha}(B_{\frac{5R}{4}})$ for some $q > d$.
Assume $u \in V_{B_{\frac{5R}{4}}}^\mu (\R^d)$ satisfies
\begin{align*}
\cE(u,\phi) &\geq (f,\phi) \text{ for any nonnegative } \phi \in
H_{B_{R}}^\mu (\R^d)\,, \\
u(x) &\geq \eps \quad \text{ for almost all } x \in B_{R} \text{ and some } \eps
> R^\delta \|f\|_{L^{q/\alpha}(B_{\frac{9R}{8}})} \,, \end{align*}
where $\delta=\alpha (\frac{q-d}{q})$. Then for $p>1$ 
\begin{align}\label{eq:moser-neg-it}
 \|u^{-1}\|^{p-1}_{L^{(p-1)\frac{d}{d-\alpha}}(B_{R})} & \leq c
\left(\max\{\tfrac{p-1}{2},\tfrac{6(p-1)^2}{16}\} \right) 
\rho^{-\alpha} \|u^{-1}\|^{p-1}_{L^{p-1}(B_{R+\rho})} \,,
\end{align}
where $c>0$ is independent of $u, x_0, R, \rho, p$, $\eps$, and $\alpha$. 
\end{lemma}

Note that the result does not require $u$ to be nonnegative in all of $\R^d$. 

\begin{proof}
Let $\tau :\R^d \to \R$ be a function 
according to 
assumption \eqref{eq:assum_cutoff}, i.e. 
\begin{align*}
\begin{cases}
\opn{supp}(\tau) \subset \overline{B_{R+\rho}} \subset 
B_{\frac{9R}{8}}, \; \|\tau\|_\infty
\leq 1\,, \forall x \in B_R:\tau(x) = 1\,, \\
\sup\limits_{x \in \R^d} \; \il_{\R^d} \big(\tau(y) - \tau(x)\big)^2  
\mu(x,\d y) \leq B \rho^{-\alpha}\,.
\end{cases}
\end{align*} 
The assumptions of the lemma imply 
\begin{align*}
\cE(u,-\tau^{p+1} u^{-p}) \; \leq \; (f, -\tau^{p+1} u^{-p})\,,
\end{align*} 
leading via \autoref{lem:thewonder} and the choice $a=u(x)$,  $b=u(y)$,
$\tau_1=\tau(x)$, $\tau_2=\tau(y)$ to 
\begin{align}\label{eq:moser-neg-1}
\begin{split}
\iil_{\R^d  \R^d}& \tau(x) \tau(y)
\Big[\big(\frac{u(y)}{\tau(y)}\big)^{\frac{-p+1}{2}} -
\big(\frac{u(x)}{\tau(x)}\big)^{\frac{-p+1}{2}} \Big]^2 \mu(x, \d y)  \d x  \\
& \leq c_1(p) \iil_{\R^d  \R^d} \big(\tau(y) - \tau(x)\big)^2 \Big[
\big(\frac{u(y)}{\tau(y)}\big)^{-p+1} + \big(\frac{u(x)}{\tau(x)}\big)^{-p+1}
\Big] \mu(x, \d y) \d x + (f, -\tau^{p+1} u^{-p})\,,
\end{split}
\end{align}
where $c_1(p)=\max\{\tfrac{p-1}{2},\tfrac{6(p-1)^2}{16}\}$. 
The left-hand side can trivially be estimated from below like this:
\begin{align*}
\begin{split}
\iil_{\R^d \R^d}& \tau(x) \tau(y)
\Big[\big(\frac{u(y)}{\tau(y)}\big)^{\frac{-p+1}{2}} -
\big(\frac{u(x)}{\tau(x)}\big)^{\frac{-p+1}{2}} \Big]^2 \mu(x, \d y)  \d x  \\
& \geq \iil_{B_R  B_R} \Big((\frac{u(y)}{\tau(y)})^{\frac{-p+1}{2}} -
(\frac{u(x)}{\tau(x)})^{\frac{-p+1}{2}} \Big)^2
\mu(x, \d y) \d x \,.
\end{split}
\end{align*}
Using symmetry, the first term on the right-hand side in
\autoref{eq:moser-neg-1} is estimated from above as follows:
\begin{align*}
\begin{split}
2c_1(p) &\iil_{\R^d  \R^d} \big(\tau(y) - \tau(x)\big)^2 \tau(x)^{p-1} u(x)
^{-p+1}  \mu(x, \d y) \d x \\
& \leq 2c_1(p) \il_{B_{R+\rho}} u(x)^{-p+1} \Big( \il_{\R^d} \big(\tau(y) -
\tau(x)\big)^2   \mu(x, \d y)  \Big) \d x \leq c_2(p) 
\rho^{-\alpha}
\il_{B_{R+\rho}} u(x)^{-p+1} \,.
\end{split}
\end{align*}
It remains to estimate $|(f, -\tau^{p+1} u^{-p})|$ from above. For any $a>0$ we
have
\begin{align*}
|(f, &-\tau^{p+1} u^{-p})| \leq \eps^{-1} |(\tau^2 f, \tau^{p-1} u^{-p+1})| \leq
\eps^{-1} \|\tau^2 f \|_{q/\alpha} 
\|\tau^{p-1} u^{-p+1} \|_{q/(q-\alpha)} \\
&=\eps^{-1} \|\tau^2 f \|_{q/\alpha} 
\|(\tau/u)^\frac{p-1}{2} \|^2_{2q/(q-\alpha)} \\
&\leq \eps^{-1} \|\tau^2 f \|_{q/\alpha} \Big\{ a \|(\tau/u)^\frac{p-1}{2}
\|^2_{2d/(d-\alpha)} 
+ a^{-d/(q-d)} \|(\tau/u)^\frac{p-1}{2} \|_2^2 \Big\} \\
&\leq (2R)^{-\alpha\frac{q-d}{q}} a \|(\tau/u)^{p-1} \|_{d/(d-\alpha)} 
+ R^{-\alpha\frac{q-d}{q}} a^{-d/(q-d)} \|(\tau/u)^{p-1} \|_1\,.
\end{align*}
We choose $a = \omega R^{\alpha\frac{q-d}{q}}$ for some $\omega$ and obtain
\begin{align*}
|(f, &-\tau^{p+1} u^{-p})| \leq \omega \|(\tau/u)^{p-1} \|_{d/(d-\alpha)} 
+ \omega^{-d/(q-d)} R^{-\alpha} \|(\tau/u)^{p-1} \|_1\,.
\end{align*}

Combining these estimates we obtain from \eqref{eq:moser-neg-1} for any $p>1$
and any $\omega >0$
\begin{align*}
\begin{split}
\iil_{B_{R+\rho}  B_{R+\rho}} & \big[(\frac{u(y)}{\tau(y)})^{\frac{-p+1}{2}} -
(\frac{u(x)}{\tau(x)})^{\frac{-p+1}{2}}\big)^2 \mu(x, \d y) \d x \\
& \leq c_3 \left(\omega^{\frac{-d}{q-d}} +
\max\{\tfrac{p-1}{2},\tfrac{6(p-1)^2}{16}\} \right) 
\rho^{-\alpha} \il_{B_{R+\rho}} u(x)^{-p+1} \, \d x  +  \omega
\|(u/\tau)^{-p+1}\|_{L^\frac{d}{d-\alpha}(B_{R+\rho})}  \,.
\end{split}
\end{align*}

Next, we use Assumption \eqref{eq:assum_comp} and apply the Sobolev
inequality, \autoref{fact:sobolev}, to the left-hand side. Choosing $\omega$
small enough and subtracting  the term $\omega
\|(u/\tau)^{-p+1}\|_{L^\frac{d}{d-\alpha}(B_{R+\rho})}$ from both sides, we
prove the assertion of the lemma. 
\end{proof}

\autoref{lem:moser-neg-it} provides us with an estimate which can be iterated.
As a result of this iteration we obtain the following corollary. 

\begin{corollary} \label{cor:inf-bound}
Assume $x_0 \in B_1$, $0 < R < 1/2$, and $0<\eta<1<\Theta$. Set
$B_{R}=B_R(x_0)$. Assume 
$f \in L^{q/\alpha}(B_{\Theta R})$ for some $q > d$. Assume $u \in 
V^\mu_{B_{\Theta R}}(\R^d)$ satisfies
\begin{align*}
\cE(u,\phi) &\geq (f,\phi) \text{ for any nonnegative } \phi \in
H^\mu_{B_{\Theta R}}(\R^d)  \\
u(x) &\geq \eps \quad \text{ for almost all } x \in B_{\Theta R} \text{ and some
} \eps > (\Theta R)^\delta 
\|f\|_{L^{q/\alpha}(B_{R \frac{1+3\Theta}{4}})} \,, 
\end{align*}
where $\delta=\alpha (\frac{q-d}{q})$. Then for any $p_0 >0$
\begin{align}\label{eq:inf-bound}
\inf\limits_{x \in B_{\eta R}(x_0)} u(x) \geq c \Big( \fint_{B_R(x_0)}
u(x)^{-p_0} \, \d x \Big)^\frac{-1}{p_0}\,,
\end{align}
where $c>0$ is independent of $u, x_0, R$, $\eps$, and $\alpha$. 
\end{corollary}

\begin{proof}
The idea of the proof is to apply \autoref{lem:moser-neg-it} to radii $R_k$,
$\rho_k$ 
with $R_k \searrow \eta R$ and $\rho_k \searrow 0$ for $k \to \infty$. For each
$k$ one chooses an exponent $p_k > 1$ with $p_k \to \infty$
for $k \to \infty$. Because of Assumption \eqref{eq:assum_comp} we
can apply the Sobolev inequality, \autoref{fact:sobolev}, to the left-hand side
in \eqref{eq:moser-neg-it}. Next,
one iterates the 
resulting inequality as in \cite{Mos61}, 
see also Chapter 8.6 in \cite{GiTr83}. The only difference to the proof in
\cite{Mos61} is that 
the factor $\frac{d}{d-2}$ now becomes $\frac{d}{d-\alpha}$. The assertion then
follows from the fact
\[ \Big(\fint_{B_{R_k}(x_0)} u^{-p_k}\Big)^\frac{-1}{p_k} \to
\inf\limits_{B_{\eta R}(x_0)} u \text{ for } k \to \infty \,. \]
\end{proof}

Let us finally prove \autoref{theo:weak_harnack_f}.

\begin{proof}[Proof of \autoref{theo:weak_harnack_f}]
Define $\ov{u}= u + \|f\|_{L^{q/\alpha}(B_{\frac{15}{16}})}$ and note that
$\sE(u,\phi)=
\sE(\ov{u},\phi)$ for any $\phi$. 
We apply \autoref{lem:moser-flip} for $R=3/4$ and obtain that there exist
$\ov{p} \in (0,1)$ and $c>0$ such that
\begin{align*}
\left(\fint_{B_{\frac34}} \ov{u}(x)^{\ov{p}} \; \d x \right)^{1/\ov{p}} \; \d x
\leq c
\left(\fint_{B_{\frac34}} u(x)^{-\ov{p}} \; \d x \right)^{-1/\ov{p}} \,.
\end{align*}
Next, we apply \autoref{cor:inf-bound} with $R=3/4$, $\eta=2/3$ and
$\Theta=4/3$. Together with the estimate from above we obtain
\begin{align}
\inf\limits_{B_\frac{1}{2}} u &\geq c \Big( \frac{1}{|B_{\frac34}|}
\il_{B_{\frac34}}
\ov{u}(x)^{\ov{p}} \, \d x \Big)^\frac{1}{\ov{p}} \,, 
\end{align} 
which, after recalling the definition of $\ov{u}$, proves 
\autoref{theo:weak_harnack_f}.
\end{proof}

\section{The weak Harnack inequality implies H\"{o}lder
estimates}\label{sec:harnackimplieshoelder}

The aim of this section is to provide the proof of 
\autoref{theo:hh-nonloc-gen-k}. As is explained in 
\autoref{subsec:intro_reg-estimates} it is well known that the Harnack 
inequality or the weak Harnack inequality imply regularity estimates in Hölder 
spaces. Here we are going to establish such a result for quite general nonlocal 
operators in the framework of metric measure spaces. 

We begin with a short study of condition \eqref{eq:nu-decay-simple}. The 
standard example that we have in mind
is given in \autoref{exa:standard_nu-xr}. Let $(X,d,m)$ be a metric measure 
space. For
$R>r>0$, $x \in
X$, set  
\begin{align}\label{eq:A-def}
B_r(x) =\{ y \in X | d(y,x) < r \}\,, \qquad A_{r,R}(x) = B_R(x) \setminus
B_r(x) \,.
\end{align}

\begin{lemma}\label{lem:nu-xr_prop} For $x \in X, r>0$ let $\nu_{x,r}$ be a
measure on $\cB(X
\setminus \{x\})$, which is finite on all sets $M$ with $\dist(\{x\},M)>0$.
Then the following conditions are equivalent:
\begin{enumerate}
 \item For some $\chi > 1$, $c \geq 1$ and all $x \in X, 0<r\leq 1, j\in \N_0$
\begin{align*}
 \nu_{x,r} (X \setminus B_{r 2^j}(x)) \leq c \chi^{-j} \,.
\end{align*}
 \item Given $\theta > 1$, there are $\chi > 1$, $c \geq 1$ such that for all $x
\in X, 0<r\leq 1, j\in \N_0$
\begin{align*}
 \nu_{x,r} (X \setminus B_{r \theta^j}(x)) \leq c \chi^{-j} \,.
\end{align*}
 \item Given $\theta > 1$, there are $\chi > 1$, $c \geq 1$ such that for all 
$x \in X, 0<r\leq 1, j\in \N_0$
\begin{align*}
 \nu_{x,r} (A_{r \theta^j, r \theta^{j+1}}(x)) \leq c \chi^{-j} \,.
\end{align*}
 \item Given $\sigma > 1, \theta > 1$ there are $\chi > 1$, $c \geq 1$ such that
for all $x \in X, 0<r\leq 1, j\in \N_0$ and $y \in B_{\frac{r}{\sigma}}(x)$ 
\begin{align}\label{eq:nu-xr_prop-sup}
 \nu_{y,r'} (A_{r \theta^j, r \theta^{j+1}}(x)) \leq c \chi^{-j} \,, \text{
where } r'=r(1-\tfrac{1}{\sigma}) \,. 
\end{align}
\end{enumerate}
If, in addition to any of the above conditions, \eqref{eq:nu_r-depend-assum}
holds, then \eqref{eq:nu-xr_prop-sup} can be replaced by 
\begin{align}\label{eq:nu-xr_prop-sup-mod}
 \nu_{y,r} (A_{r \theta^j, r \theta^{j+1}}(x)) \leq c \chi^{-j}  \,. 
\end{align}
\end{lemma}
\begin{proof} In $\theta > 2$, the implication (1)$\Rightarrow$(2) trivially
holds true. For $\theta < 2$ it can be obtained by adjusting $\chi$
appropriately. The proof of (2)$\Rightarrow$(1) is analogous. The implication
(2)$\Rightarrow$(3) trivially holds true. The implication (3)$\Rightarrow$(2)
follows from
\[  \nu_{x,r} (X \setminus B_{r \theta^j}(x)) = \suml_{k=j}^{\infty} \nu_{x,r}
(A_{r \theta^k, r \theta^{k+1}}(x)) \leq c \suml_{k=j}^{\infty}
\chi^{-k} = c (\tfrac{\chi}{\chi-1}) \chi^{-j} \,.\]
The implication (4)$\Rightarrow$(3) trivially holds true. Instead of
(3)$\Rightarrow$(4) we explain the proof of (2)$\Rightarrow$(4). Fix $\sigma >
1, \theta > 1, x \in X, r>0, j\in \N_0$ and $y \in B_{\frac{r}{\sigma}}(x)$.
Set $r'=r(1-\tfrac{1}{\sigma})$. Then $X \setminus B_{r\theta^j}(x) \subset X
\setminus B_{r'\theta^j}(y)$. Thus 
\[   \nu_{y,r'} (X \setminus B_{r \theta^j}(x)) \leq \nu_{y,r'} (X
\setminus B_{r' \theta^j}(y)) \leq c \chi^{-j} \,. \]
\end{proof}

\begin{remark}
Note that the conditions above imply that, given $j \in \N_0$ and $x \in X$, the
quantity $ \limsup\limits_{r \to 0+} \nu_{x,r} (X \setminus B_{r 2^j}(x))$ is
finite. 
\end{remark}

\begin{remark}\label{rem:nu-xr_prop-sup} Let $x \in X, A \in \cB(X
\setminus \{x\})$
with $\dist(\{x\},A)>0$. 
In the applications that are of interest to us, the function $r \mapsto
\nu_{x,r}(A)$ is strictly increasing with $\nu_{x,0}(A) = 0$. 
\end{remark}

\begin{proof}[Proof of \autoref{theo:hh-nonloc-gen-k}]
The proof follows closely the strategy of \cite{Mos61}, see also \cite{Sil06}.
In the sequel of the proof,
let us write $B_t$ instead of $B_t(x_0)$ for $t>0$.
Fix $r\in(0,r_0)$ and $u\in \cS_{x_0,r}$.
Let $c_1 \geq 1$ be the constant
in \eqref{eq:hoelder-harnack-lem}. Set $\kappa = (2 c_1 2^{1/p})^{-1}$ and
\[ \beta = \ln(\tfrac{2}{2{-}\kappa}) / \ln(\theta) \qquad \Rightarrow
(1-\tfrac{\kappa}{2}) = \theta^{-\beta} \,. \]
Set $M_0=\|u\|_\infty$, $m_0 = \inf\limits_{X} u(x)$ and $M_{-n}=M_0$,
$m_{-n}=m_0$ for $n \in \N$. We will construct an increasing sequence
$(m_n)$ and a decreasing sequence $(M_n)$ such that for $n \in \Z$
\begin{align}\label{eq:hh-nonloc-1L}
\begin{split}
&m_n \leq u(z) \leq M_n \quad \text{ for almost all } z \in B_{r \theta^{-n}}
\,, \\
&M_n - m_n \leq K \theta^{-n\beta} \,,
\end{split}
\end{align}
where $K=M_0-m_0\in [0,2\|u\|_\infty]$. Assume there is $k \in \N$ and there are
$M_n,
m_n$ such that \eqref{eq:hh-nonloc-1L} holds for $n \leq k-1$. We need to
choose
$m_k, M_k$ such that \eqref{eq:hh-nonloc-1L} still holds for $n=k$. Then the
assertion of the lemma follows by complete induction. For $z \in X$ set
\[ v(z) = \Big( u(z) - \frac{M_{k-1}+m_{k-1}}{2} \Big) \frac{2
\theta^{(k-1)\beta}}{K} \,. \]
The definition of $v$ implies $v \in \cS_{x_0,r}$ and $|v(z)|
\leq 1$ for almost any $z \in B_{r\theta^{-(k-1)}}$. Our next aim is to show
that \eqref{eq:hoelder-harnack-lem} implies that either $v \leq 1-\kappa$
or $v \geq -1 + \kappa$ on $B_{r\theta^{-k}}$. Since our version of the
Harnack inequality contains nonlocal terms we need to
investigate the
behavior of $v$ outside of $B_{r\theta^{-(k-1)}}$. Given $z \in
X$ with $d(z,x_0) \geq r\theta^{-(k-1)}$ there is $j \in \N$ such
that
\[ r\theta^{-k+j} \leq d(z,x_0) < r\theta^{-k+j+1}   \,. \]
For such $z$ and $j$ we conclude
{\allowdisplaybreaks
\begin{align}
\frac{K}{2 \theta^{(k-1)\beta}} v(z) &=  \Big( u(z) - \frac{M_{k-1}+m_{k-1}}{2}
\Big) \leq \Big(M_{k-j-1} - m_{k-j-1} + m_{k-j-1} -
\frac{M_{k-1}+m_{k-1}}{2} \Big)
\nonumber\\
&\leq \Big(M_{k-j-1} - m_{k-j-1} - \frac{M_{k-1}-m_{k-1}}{2}
\Big) \leq \Big(K
\theta^{-(k-j-1)\beta} - \tfrac{K}{2} \theta^{-(k-1)\beta} \Big) \,, 
\nonumber\\
\text{ i.e. }  v(z) &\leq 2 \theta^{j\beta} - 1 \; \leq 2 \Big(\theta
\frac{d(z,x_0)}{r \theta^{-(k-1)}}\Big)^\beta -1  \,, 
\label{eq:hh-nonloc-above-est}
\end{align}
}and
{\allowdisplaybreaks
\begin{align*}
\frac{K} {2 \theta^{(k-1)\beta}} v(z) &=  \Big( u(z) - \frac{M_{k-1}+m_{k-1}}{2}
\Big) \geq  \Big(m_{k-j-1} - M_{k-j-1} + M_{k-j-1} - \frac{M_{k-1}+m_{k-1}}{2}
\Big)
\\
&\geq \Big(-\big(M_{k-j-1} - m_{k-j-1}\big) + \frac{M_{k-1}-m_{k-1}}{2}
\Big) \geq \Big(-K \theta^{-(k-j-1)\beta} + \tfrac{K}{2} \theta^{-(k-1)\beta}
\Big)
\,,  \\ 
\text{ i.e. } v(z) &\geq 1- 2 \theta^{j\beta} \; \geq 1- 2 \Big(\theta
\frac{d(z,x_0)}{r \theta^{-(k-1)}}\Big)^\beta  \,.
\end{align*}
}Now there are two cases:

Case 1: $m(\{x \in B_{r\theta^{-k+1}/\lambda}| v(x) \leq 0 \}) \geq \frac12
m(B_{r\theta^{-k+1}/\lambda})$

Case 2: $m(\{x \in B_{r\theta^{-k+1}/\lambda}| v(x) >  0 \}) \geq \frac12
m(B_{r\theta^{-k+1}/\lambda})$

We work out details for Case 1 and comment afterwards on Case 2.
In Case 1 our aim is to show $v(z) \leq 1-\kappa$ for almost every $z \in
B_{r\theta^{-k}}$ and some $\kappa \in (0,1)$. Because then for
almost any $z \in B_{r\theta^{-k}}$
\begin{align} \label{eq:hh-nonloc-lem-3A}
\begin{split}
u(z) &\leq  \tfrac{(1-\kappa) K}{2} \theta^{-(k-1)\beta} +
\frac{M_{k-1}+m_{k-1}}{2} \\
&= \tfrac{(1-\kappa) K}{2} \theta^{-(k-1)\beta} +
\frac{M_{k-1}-m_{k-1}}{2} + m_{k-1}\\
&= m_{k-1} + \tfrac{(1-\kappa) K}{2} \theta^{-(k-1)\beta} + \tfrac12 K 
\theta^{-(k-1)\beta}\\
&\leq m_{k-1} + K \theta ^{-k\beta}\,.
\end{split}
\end{align}
We then set $m_k=m_{k-1}$ and $M_k= m_{k} + K \theta^{-k\beta}$ and obtain,
using \eqref{eq:hh-nonloc-lem-3A}, $m_k \leq u(z) \leq M_k$ for almost every $z
\in
B_{r\theta^{-k}}$, what needs to be proved.

Consider $w=1-v$ and note $w \in \cS_{x_0,r\theta^{-(k-1)}}$ and $w
\geq 0$ in $B_{r\theta^{-(k-1)}}$. We apply
\eqref{eq:hoelder-harnack-lem} and obtain
\begin{align}\label{eq:harnack_applied-lem}
\Big( \fint\limits_{B_{r\theta^{-k+1}/\lambda}(x_0)} w^p \d m \Big)^{1/p}
\leq c_1
\inf\limits_{B_{r\theta^{-k}}} w 
  + c_1 \sup\limits_{x\in B_{r\theta^{-k+1}/\sigma}} \il_{X} w^-(z)
\nu_{x, r\theta^{-(k-1)}}(\d z) \,,
\end{align}
In Case~1 the left-hand side of \eqref{eq:harnack_applied-lem} is bounded from 
below
by $(\frac12)^{1/p}$.
This, the estimate \eqref{eq:hh-nonloc-above-est} on $v$ from above leads to
\begin{align*}
\inf\limits_{B_{r\theta^{-k}}} w &\geq (c_1 2^{1/p})^{-1} 
- \sup\limits_{x\in B_{r\theta^{-k+1}/\sigma}} \il_{X} w^-(z) 
\nu_{x,r\theta^{-(k-1)}}(\d z) \\
&\geq (c_1 2^{1/p})^{-1} - \suml_{j=1}^\infty
 \sup\limits_{x\in B_{r\theta^{-k+1}/\sigma}} \il
\mathbbm{1}_{A_{r\theta^{-k+j},r\theta^{-k+j+1}} (x_0)} (1-v(z))^-
\, \nu_{x,r\theta^{-(k-1)}}(\d z) \\
&\geq (c_1 2^{1/p})^{-1} - \suml_{j=1}^\infty
(2\theta^{j\beta}-2) \eta_{x_0,r, \theta,j,k} \,,
\end{align*}
where $\eta_{x_0,r, \theta,j,k} = \sup\limits_{x\in B_{r\theta^{-k+1}/\sigma}} 
\nu_{x,r\theta^{-(k-1)}}(A_{r\theta^{-k+j},r\theta^{-k+j+1}} (x_0))$. Now, 
\eqref{eq:nu-xr_prop-sup-mod} implies that $\eta_{x_0,r, \theta,j,k} \leq c 
\chi^{-j-1}$. Thus we obtain

\begin{align}
\inf\limits_{B_{r\theta^{-k}}} w \geq (c_1 2^{1/p})^{-1} - 2 c
\suml_{j=1}^\infty (\theta^{j\beta}-1) \chi^{-j-1} \,.
\end{align}

Note that $\suml_{j=1}^\infty \theta^{j\beta}
\chi^{-j-1} < \infty$ for $\beta>0$ small enough, i.e. there
is $l \in \N$ with
\[ \suml_{j=l+1}^\infty (\theta^{j\beta}-1) \chi^{-j-1} \leq 
\suml_{j=l+1}^\infty \theta^{j\beta} \chi^{-j-1} \leq (16 c_1)^{-1} \,. \]
Given $l$ we choose $\beta >0$ smaller (if needed) in order to assure
\[ \suml_{j=1}^{l} (\theta^{j\beta}-1) \chi^{-j-1} \leq (16 c_1)^{-1} \,. \]
The number $\beta$ depends only on $c_1$, $c$, $\chi$ from 
\eqref{eq:nu-xr_prop-sup-mod} and on
$\theta$.
Thus we have shown that $w
\geq \kappa$ on $B_{r\theta^{-k}}$ or, equivalently, $v \leq
1-\kappa$ on $B_{r\theta^{-k}}$.

In Case 2 our aim is to show $v(x) \geq -1+\kappa$. This time, set $w=1+v$.
Following
the strategy above one sets $M_k=M_{k-1}$ and $m_k= M_{k} - K
\theta^{-k\beta}$ leading to the desired result.

Let us show how \eqref{eq:hh-nonloc-1L} proves the assertion of the
lemma.
Given $\rho \leq r$, there exists $j\in \N_0$ such that
\[
 r\theta^{-j-1} \leq \rho \leq r\theta^{-j}.
\]
From \eqref{eq:hh-nonloc-1L} we conclude
\[
 \osc_{B_\rho} u \leq \osc_{B_{r\theta^{-j}}} u \leq M_j - m_j
 \leq 2\theta^\beta \|u\|_\infty \left( \frac{\rho}{r} \right)^\beta.\qedhere
\]
\end{proof}

\begin{corollary}\label{cor:hh-nonloc-gen-k}
Let $\Omega=B_{r_0}(x_0) \subset X$ and let $\sigma, \theta, \lambda > 1$.
Let $\cS_{x,r}$ and $\nu_{x,r}$ be as
above. Assume that conditions \eqref{eq:nu-decay-simple},
\eqref{eq:nu_r-depend-assum} are satisfied. Assume that there is $c\geq 1$ such 
that
for $0 < r \leq r_0$,
\begin{align}\label{eq:hoelder-harnack-ass}
\left.
\begin{cases}
\big(B_r(x) \subset \Omega\big) \wedge
\big(u \in \cS_{x,r}) \wedge \big(u \geq 0 \text{ in } B_r(x) \big) \,, \\
\quad \Rightarrow \quad
\Big( \fint\limits_{B_{\frac{r}{\lambda}}(x)} u(\xi)^p m(\d \xi) \Big)^{1/p} 
\leq
c \inf\limits_{B_{\frac{r}{\theta}}(x)} u + c  \sup\limits_{\xi \in
B_{\frac{r}{\sigma}}(x)} \il_{X} u^-(z) \nu_{\xi,r} (\d z) \,.
\,.
\end{cases} \right\}
\end{align}
Then there exist $\beta \in (0,1)$ such that for
every $u\in \cS_{x_0,r_0}$  and almost every $x$, $y \in \Omega$
\begin{align} \label{eq:hoelder-asser-abstr}
|u(x) - u(y)| \leq 16\theta^\beta \|u\|_\infty \Big( 
\frac{d(x,y)}{d(x,\Omega^c) \vee d(y,\Omega^c)}
\Big)^\beta
\, .
\end{align}
\end{corollary}

\begin{proof} By symmetry, we may assume that $r:=d(y,\Omega^c) \geq 
d(x,\Omega^c)$.
Furthermore, it is enough to prove \eqref{eq:hoelder-asser-abstr} for
pairs $x$, $y$ such that $d(x,y) < r/8$, as in the opposite case the assertion
is obvious.

We fix a~number $\rho \in (0,r_0/4)$ and consider all pairs of $x$, $y \in
\Omega$ such that
\begin{equation}\label{lem:pairs-xy}
  \frac{\rho}{2} \leq d(x,y) \leq \rho.
\end{equation}
We cover the ball $B_{r_0 - 4\rho}(x_0)$ by a countable family of balls
$\tilde{B}_i$ with radii $\rho$.
Without loss of generality, we may assume that  $\tilde{B}_i \cap B_{r_0 -
4\rho}(x_0) \neq \emptyset$.
Let $B_i$ resp. $B_i^*$ denote the balls with the same center as the ball
$\tilde{B}_i$ and the radius $2\rho$ resp. the maximal radius that allows for 
$B_i^* \subset \Omega$.

Let $x$, $y\in \Omega$ satisfy \eqref{lem:pairs-xy}. From $r> 8d(x,y) \geq 
4\rho$ it follows that $y \in B_{r_0 - 4\rho}(x_0)$,
therefore
$y\in \tilde{B}_i$ for some index $i$.
We observe that both $x$ and $y$ belong to $B_i$. We
apply \autoref{theo:hh-nonloc-gen-k} to 
$x_0$ and $r_0$ being the center and radius of $B_i^*$, respectively,
and obtain
\begin{align*}
 \osc_{B_i} u &\leq 2\theta^\beta  \|u\|_\infty
\left(\frac{\radius(B_i)}{\radius(B_i^*)}\right)^\beta
\leq 2\theta^\beta  \|u\|_\infty \left(\frac{\rho}{r-\rho}\right)^\beta\\
&\leq \frac{16}{3}\theta^\beta  \|u\|_\infty 
\left(\frac{d(x,y)}{r}\right)^\beta.
\end{align*}
Hence \eqref{eq:hoelder-asser-abstr} holds, provided $x$ and $y$ are such that
$|u(x)-u(y)| \leq \osc_{B_i} u$.

By considering $\rho=r_0 2^{-j}$ for $j=3$, $4$, \ldots, we prove 
\eqref{eq:hoelder-asser-abstr}
for almost all $x$ and $y$ such that $d(x,y) \leq r_0/8$,
 hence the proof is finished.
\end{proof}

\subsection{Proof of \autoref{theo:reg_result_nonlocal}}
\label{sec:regularity}

We are now going to use the above results and prove one of our main results. 

\begin{proof}[Proof of \autoref{theo:reg_result_nonlocal}]
The proof of \autoref{theo:reg_result_nonlocal} follows from 
\autoref{cor:weak_harnack} and \autoref{cor:hh-nonloc-gen-k}. 
The proof is complete once we can apply 
\autoref{cor:hh-nonloc-gen-k} for $x_0 = 0$ und $r_0 = \frac12$. Assume $0 < r 
\leq r_0$ and $B_r(x) \subset B_{\frac12}$. Let 
$\cS_{x,r}$ be the set of all functions $u \in V^\mu_{B_{r}(x)}(\R^d)$ 
satisfying $\cE(u,\phi) = 0$ for every $\phi \in H^\mu_{B_{r}(x)}(\R^d)$. 
Assume $u \in \cS_{x,r}$ and $u \geq 0$ in $B_{r}(x)$. Then 
\autoref{cor:weak_harnack} implies
\[   \inf\limits_{B_{\frac{r}{4}}(x)} u \geq c \big(
\fint\limits_{B_{\frac{r}{2}}(x)}
u(x)^{p_0} \, \d x \big)^{1/p_0} - r^\alpha \sup\limits_{y \in
B_{\frac{15 R}{16}}(x)}
\int\limits_{\R^d \setminus B_r(x)} u^-(z) \mu(y , \d z) \,, \] 
with positive constants $p_0, c$ which depend only on  $d, \alpha_0, A, 
B$. Set $\theta=4, \lambda=2, \sigma=\frac{16}{15}$. Let $\nu_{x,r}$ 
be the measure on
$\R^d \setminus B_r(x)$ defined by 
\[
 \nu_{x,r}(A) = r^\alpha \mu(x, A) 
\] 
The condition \eqref{eq:nu_r-depend-assum} obviously holds true. The 
condition \eqref{eq:nu-decay-simple} follows from \eqref{eq:assum_largejumps}. 
Thus we can apply \autoref{cor:hh-nonloc-gen-k} for $x_0 = 0$ und $r_0 = 
\frac12$ and obtain the assertion of \autoref{theo:reg_result_nonlocal}. The 
proof is complete. 
\end{proof}

\section{Local comparability results for nonlocal quadratic 
forms}\label{sec:comparability}

The aim of this section is to prove \autoref{theo:U1L1implyA}. The assertion of 
this result is that \eqref{eq:assum_comp} and \eqref{eq:assum_cutoff} hold 
true under certain assumptions on $\mu(\cdot, \d y)$, 
see \autoref{subsec:intro_comp}. 
It is easy to prove that \eqref{eq.lo-up-meas} and \eqref{eq:cond_U} imply (B) 
with a constant $B \geq 1$ independent of $\alpha \in (\alpha_0,2)$: Let 
$\tau\in C^\infty(\R^d)$ be a
function satisfying
$\supp(\tau)=\overline{B_{R+\rho}}$,
$\tau\equiv 1$ on $B_R$, $0\leq \tau \leq 1$ on $\R^d$ and $|\tau(x)-\tau(y)|
\leq 2\rho^{-1}|x-y|$ for all $x$, $y\in \R^d$.
In particular, we have then $|\tau(x)-\tau(y)| \leq (2\rho^{-1}|x-y|) \wedge 1$.
For every $x\in \R^d$ we obtain
\begin{align*}
\int_{\R^d} (\tau(x)-\tau(y))^2 \mu(x,dy)
 &\leq
\int_{\R^d} \left( (4\rho^{-2}|z|^2) \wedge 1 \right) \upB(dz)\\
 &= 4\rho^{-2}
\int_{\R^d} (|z|^2 \wedge \frac{\rho^2}{4}) \upB(dz) \leq 2^\alpha C_U
\rho^{-\alpha} \leq 4 C_U \rho^{-\alpha}.\qedhere
\end{align*}

Thus we only need to concentrate on proving \eqref{eq:assum_comp}. The upper 
bound can be established quite easily, so we do this first.

\subsection{Upper bound in \texorpdfstring{\eqref{eq:assum_comp}}{(A)}}

Let us formulate and prove the following
comparability result.

\begin{proposition}\label{prop:upper}
Assume that $\nu$ satisfies \eqref{eq:cond_U} with the constant $C_U$ and  let
$0<\alpha_0 \leq \alpha<2$.
If $D\subset \R^d$ is a bounded Lipschitz domain, then
there exists a constant $c=c(\alpha_0,d,C_U, D)$ such that
\begin{equation}\label{Dupper}
 \cE^\nu_D(u,u) \leq c \cE^{\mu_\alpha}_D(u,u), \quad u\in H^{\alpha/2}(D).
\end{equation}
The constant $c$ may be chosen such that \eqref{Dupper} holds for all balls
$D=B_r$ of radius $r<1$, and for all $\alpha\in [\alpha_0,2)$.
\end{proposition}
\begin{proof}
By $E$ we denote the extension operator from $H^{\alpha/2}(D)$ to
$H^{\alpha/2}(\R^d)$, see \autoref{fact:extension}.
By subtracting  a constant, we may and do assume that $\int_D u\,\d x = 0$.
We have by Plancherel formula and Fubini theorem
{\allowdisplaybreaks
\begin{align}
\cE^\nu_{D}(u,u) &=
\int_D\!\int_{D-y} (u(y+z)-u(y))^2 \,\nu(dz)\,\d y \label{Plproof}\\
&\leq
\int_D\!\int_{B_{\diam D}(0)} (Eu(y+z)-Eu(y))^2 \,\nu(dz)\,\d y \nonumber\\
&\leq
\int_{B_{\diam D}(0)}\!\int_{\R^d} (Eu(y+z)-Eu(y))^2 \,\d y \, \,\nu(dz)
\nonumber\\
 &=\int_{\R^d} \left( \int_{B_{\diam D}(0)} |e^{i\xi\cdot z}-1|^2 \,\nu(dz)
\right)
|\widehat{Eu}(\xi)|^2\,d\xi  \nonumber\\
 &=\int_{\R^d} \left( \int_{B_{\diam D}(0)}4\sin^2\Big(\frac{\xi\cdot
z}{2}\Big) 
\,\nu(dz) \right) |\widehat{Eu}(\xi)|^2\,d\xi.\label{Plproof2}
\end{align}
}
For  $|\xi| > 2$ we obtain, using \eqref{eq:cond_U}
\begin{equation}\label{largexi} 
 \int 4\sin^2\Big(\frac{\xi\cdot z}{2}\Big) \,\nu(dz) \leq
|\xi|^2 \int (|z|^2 \wedge 4|\xi|^{-2}) \,\nu(dz) 
\leq 4C_U |\xi|^\alpha,
\end{equation}
and for $|\xi| \leq 2$
\begin{align*}
 \int 4\sin^2\Big(\frac{\xi\cdot z}{2}\Big) \,\nu(dz) &\leq
4 \int \left(\Big|\frac{\xi\cdot z}{2}\Big|^2 \wedge 1\right) \,\nu(dz) 
\leq 4C_U.
\end{align*}
Thus
\begin{align}
\cE^\nu_{D}(u,u) &\leq
c' \int_{\R^d} \left(|\xi|^\alpha + 1  \right)
|\widehat{Eu}(\xi)|^2\,d\xi \nonumber\\
& \leq c' \|Eu\|_{H^{\alpha/2}(\R^d)}^2 \leq c
\|u\|_{H^{\alpha/2}(D)}^2
=c (\cE^{\mu_\alpha}_D(u,u) + \|u\|_{L^2(D)}^2)\label{EkdEa}
\end{align}
with  $c=c(d,C_U, D)$.
Since $\int_D u\,\d x=0$, we have by \autoref{fact:poincare-i}
\begin{align*}
\cE^{\mu_\alpha}_D(u,u) &\geq c(\alpha_0, d, D)  \int_D u^2(x)\,\d x
\end{align*}
and this together with \eqref{EkdEa} proves \eqref{Dupper}.

By scaling, the last assertion of the Theorem is satisfied with a~constant
$c=c(\alpha_0,d, C_U, B_1)$.
\end{proof}

\begin{proof}[Proof of \autoref{theo:U1L1implyA}  -- upper bound in
\eqref{eq:assum_comp}]
The second inequality in \eqref{eq:assum_comp} follows from
\autoref{prop:upper}.
We note that the constant in this inequality is robust under the mere assumption
that $\alpha$ is bounded away from zero.
\end{proof}

\subsection{Lower bound in \texorpdfstring{\eqref{eq:assum_comp}}{(A)}}

The main difficulty in establishing the lower bound in \eqref{eq:assum_comp} is 
that the measures might be singular. We will introduce a new convolution-type 
operation that, on the one hand, smoothes the support of the measures and, on 
the other hand, interacts nicely with our quadratic forms. The main result of 
this subsection is \autoref{prop:Hausdorff}.

For $\lambda < 1 \leq \eta$ and $\alpha \in (0,2)$ let
\begin{align}\label{def:g_lam_eta}
 g_\lambda^\eta(y,z) = \frac{1}{2-\alpha} |y+z|^{\alpha} 
\mathbbm{1}_{A_{|y+z|}}(y)
   \mathbbm{1}_{A_{|y+z|}}(z), \qquad y,z\in \R^d,
\end{align}
where
\[
 A_r = B(0,\eta r) \setminus B(0,\lambda r).
\]

\begin{definition}\label{def:heart_convolution} For measures $\nu_1, \nu_2$ on 
  $\cB(\R^d)$ satisfying \eqref{eq:cond_U} with some $\alpha\in (0,2)$, define a new 
measure $\nu_1 \heartsuit \nu_2$ on $\cB(\R^d)$ by
\[
 \nu_1 \heartsuit \nu_2(E) = \iint  \mathbbm{1}_{E\cap B_2}(\eta (y+z))
g_\lambda^\eta(y,z)
\,\nu_1(dy)\,\nu_2(dz),
\]
i.e.,
\[
 \int f(x) \nu_1 \heartsuit \nu_2(dx) = \iint (f\cdot \1_{B_2})(\eta(y+z)) 
g_\lambda^\eta(y,z)
\,\nu_1(dy)\,\nu_2(dz),
\]
for every measurable function $f:\R^d\to [0,\infty]$. 
\end{definition}

This definition is tailored for our applications and needs some 
explanations. We consider $\nu_1 \heartsuit \nu_2$ only for measures 
$\nu_j$, which satisfy  \eqref{eq:cond_U} with some $\alpha\in (0,2)$ for $j 
\in \{1,2\}$. This $\alpha$ equals the exponent $\alpha$ in the definition of 
$g_\lambda^\eta$. The above definition does not require $\nu_j$ to satisfy 
\eqref{eq:cond_scaling} but most often, this will be the case. Note that 
\autoref{def:heart_convolution} is valid for any choice $\lambda < 1 \leq 
\eta$. However, it will be important to choose $\lambda$ small enough and 
$\eta$ large enough. The precise bounds depend on the number $a$ from 
\eqref{eq:cond_scaling}, see \autoref{prop:Hausdorff}. Before we explain and 
prove the rather technical details, let us treat an example. 

Let us study \autoref{exa:axes} in $\R^2$. Assume $\alpha \in (0,2)$ and 
\begin{align*}
\nu_1(\d h) &= (2-\alpha) |h_1|^{-1-\alpha} \d h_1 
\delta_{\{0\}}(\d h_2)\,, \\
\nu_2(\d h) &= (2-\alpha) |h_2|^{-1-\alpha} \d 
h_2 \delta_{\{0\}}(\d h_1) \,.
\end{align*}
Both measures are one-dimensional $\alpha$-stable measures which are orthogonal 
to each other. The factor $(2-\alpha)$ ensures that for $\alpha \to 2-$ the 
measures do not explode. Let us show that $\nu_1 \heartsuit \nu_2$ is 
already absolutely continuous with respect to the two-dimensional Lebesgue 
measure.  For $E \subset B_2$; by the \autoref{def:heart_convolution} and 
the Fubini theorem 
\begin{align*}
\nu_1 &\heartsuit \nu_2(E) \\
&= (2-\alpha) \iiiint  |y+z|^{\alpha} 
\mathbbm{1}_E(\eta(y+z))
\mathbbm{1}_{A_{|y+z|}}(y)
   \mathbbm{1}_{A_{|y+z|}}(z) |y_1|^{-1-\alpha} 
|z_2|^{-1-\alpha} \ldots \\
&\qquad \ldots  \delta_{\{0\}}(\d y_2) \, \delta_{\{0\}}(\d z_1) \d y_1 
\d 
z_2 \\
&= (2-\alpha) \iint  |(y_1, z_2)|^{\alpha} \mathbbm{1}_E(\eta(y_1,z_2))
\mathbbm{1}_{A_{|(y_1, z_2)|}}(y_1,0)
   \mathbbm{1}_{A_{|(y_1, z_2)|}}(0,z_2) |y_1|^{-1-\alpha} 
|z_2|^{-1-\alpha} \, \d y_1 \d 
z_2 \\
&= (2-\alpha) \iint   \mathbbm{1}_E(\eta x)
\mathbbm{1}_{A_{|x|}}(x_1,0)
   \mathbbm{1}_{A_{|x|}}(0,x_2) |x|^{\alpha} |x_1|^{-1-\alpha} 
|x_2|^{-1-\alpha} \, \d x_1 \d 
x_2.
\end{align*}
The above computation 
shows that the measure $\nu_1 \heartsuit \nu_2$ is absolutely continuous with 
respect to the two-dimensional Lebesgue measure,
because  $\nu_1 \heartsuit \nu_2(\R^d\setminus B_2)=0$.
Let us look at the density 
more closely.

So far, we have not specified $\lambda$ and $\eta$ in the definition of 
$g_\lambda^\eta$. If $\lambda < 1$ is too large (in this particular case, if $\lambda>1/\sqrt{2}$), 
then $\mathbbm{1}_{A_{|x|}}(x_1,0) \mathbbm{1}_{A_{|x|}}(0,x_2) = 0$ for all $x 
\in \R^2$. If $\lambda$ is sufficiently small, then the support of the function 
$\mathbbm{1}_{A_{|x|}}(x_1,0) \mathbbm{1}_{A_{|x|}}(0,x_2)$ is a double-cone 
centered around the diagonals $\{x \in \R^2 | |x_1|=|x_2|\}$. Let us denote 
this support by $M$. Note that on $M$ the function $|x|^{\alpha} 
|x_1|^{-1-\alpha} 
|x_2|^{-1-\alpha}$ is comparable to $|x|^{-2-\alpha}$. Thus indeed the quantity 
$\nu_1 \heartsuit \nu_2$ is comparable to an $\alpha$-stable measure in 
$\R^2$. If we continue the procedure and define 
\[ \widetilde{\nu} = (\nu_1 \heartsuit \nu_2) \heartsuit (\nu_1 \heartsuit 
\nu_2) \,, \]
then we can make use of the fact that $(\nu_1 \heartsuit
\nu_2)$ is already absolutely continuous with respect to the two-dimensional 
Lebesgue measure. Note that, if 
$\mu_j = h_j\, dx$, then $\mu_1 \heartsuit \mu_2$ has a~density $h_1\heartsuit 
h_2$
with respect to the Lebesgue measure given by
\begin{equation}\label{e.L-density}
 h_1\heartsuit h_2(\eta y) 
= \frac{\eta^{-d} |y|^\alpha}{2-\alpha} \int \1_{A_{|y|}}(y-z) \1_{A_{|y|}}(z) 
h_1(y-z) h_2(z) \,dz,
\qquad \eta y\in B_2.
\end{equation}
In this way we conclude that $\widetilde{\nu}$ has full support and is 
comparable to a rotationally symmetric $\alpha$-stable measure in $\R^2$. With 
this observation we end our study of \autoref{def:heart_convolution} in light 
of \autoref{exa:axes}.

Before we proceed to the proofs, let us  informally explain the idea 
behind \autoref{def:heart_convolution} and our strategy. In the inner integral 
defining
\[
 \cE^{\nu}_B(u,u) = \int_B \int_{\R^d} \big(u(x)-u(x+h)\big)^2 \1_B(x+h) \,
\nu(\d h)\, \d x
\]
we take into account squared increments $(u(x)-u(x+h))^2$ in these directions
$h$, which are charged by the measure $\nu$ and such that $x+h$ is still in $B$.
By changing the variables, we see that we also have squared increments
$(u(x+h)-u(x+h+z))^2$,
 again in directions $z$, which are charged by the measure $\nu$ and such that
$x+h+z$ is still in $B$. This allows us to estimate the integral
$\cE^{\nu}_B(u,u)$ from below by a~similar integral with $\nu$ replaced by some
kind of a~convolution of $\nu$ with itself.
Measure $\nu \heartsuit \nu$ turns out to be the right convolution for this
purpose, see \autoref{lem:heart-form-comp}.

In the definition of $\nu \heartsuit \nu$, function $g_\lambda^\eta$ vanishes if
$|y|$ or
$|z|$ is bigger than $\eta|y+z|$ or smaller than $\lambda |y+z|$.
This means, in our interpretation, that we consider only those pairs of jumps
which are comparable with the size of the whole two-step jump (and in
particular,
the jumps must be comparable with each other).

To conclude these informal remarks on the definition of $\nu_1\heartsuit \nu_2$
let us note that if $\nu_1$ and $\nu_2$ have 'good properties', then so has
$\nu_1\heartsuit \nu_2$ (see \autoref{lem:heart-scaling}
and \autoref{lem:heart-upper}) and that $\cE_B^{\nu_1\heartsuit \nu_2}(u,u)$ 
can be
estimated from above by $\cE_B^{\nu_j}(u,u)$ (see
\autoref{lem:heart-form-comp}).
This allows us to reduce the problem of estimating $\cE_B^\nu(u,u)$ from below
to estimating $\cE_B^{\nu\heartsuit\nu} (u,u)$ from below, and this turns out to
be easier, since the $\heartsuit$-convolution makes the measure more 'smooth',
see \autoref{prop:Hausdorff}.

\begin{lemma}\label{lem:heart-scaling}
If two measures $\nu_j$ for $j \in \{1,2\}$ satisfy the scaling assumption 
\eqref{eq:cond_scaling} for some $a>1$,
then so does the measure $\nu_1 \heartsuit \nu_2$ for the same constant $a$.
\end{lemma}
\begin{proof}
If $\supp f\subset B_1$, then
\begin{align*}
\int f(ax)\nu_1\heartsuit \nu_2(dx) &=
\iint f(\eta a(y+z)) \1_{B_2}(\eta(y+z)) g_\lambda^\eta(y,z) 
\,\nu_1(dy)\,\nu_2(dz) \\
&=
a^{-\alpha} \iint f(\eta(ay+az)) g_\lambda^\eta(ay,az) \,\nu_1(dy)\,\nu_2(dz),
\end{align*}
because $g_\lambda^\eta(y,z) =a^{-\alpha} g_\lambda^\eta(ay,az)$.  
We observe that the function $(y,z) \mapsto f(\eta(y+z)) g_\lambda^\eta(y,z)$ 
vanishes
outside $B_1\times B_1$. Hence we may apply \eqref{eq:cond_scaling} twice to
obtain
\[
\int f(ax)\nu_1\heartsuit \nu_2(dx)
=
a^\alpha \iint f(\eta(y+z)) g_\lambda^\eta(y,z) \,\nu_1(dy)\,\nu_2(dz) = 
a^\alpha \int
f(x)\nu_1\heartsuit \nu_2(dx). \qedhere
\]
\end{proof}

Next, we establish conditions which are equivalent to \eqref{eq:cond_U}.
We say that a measure $\nu$ on $\cB(\R^d)$ satisfies the 
upper-bound assumption 
\eqref{eq:cond_U0} if for some $C_0 >0$
\begin{equation}\label{eq:cond_U0}\tag{U0}
\int_{\R^d} (|z|^2 \wedge 1)\, \nu(\d z) \leq C_0  \,. 
\end{equation}
We say that a measure $\nu$ on $\cB(\R^d)$ satisfies the upper-bound assumption 
\eqref{eq:cond_U1'} if there exists $C_1 >0$  such that for every $r \in (0,1)$
\begin{equation}\label{eq:cond_U1'}\tag{U1}
 \int_{B_r(0)} |z|^2 \,\nu(dz) \leq C_1 r^{2-\alpha} \,.
\end{equation}

\begin{lemma}\label{lem:U1p}
\begin{align*}
 \eqref{eq:cond_U} \quad \Longleftrightarrow \quad \eqref{eq:cond_U0} \wedge
\eqref{eq:cond_U1'} \,.
\end{align*}
If the constants $C_0, C_1$ are independent of $\alpha \in 
[\alpha_0,2)$, then so is $C_U$,
and vice versa.
\end{lemma}

\begin{proof}
The implications \eqref{eq:cond_U} $\Rightarrow$ \eqref{eq:cond_U1'} and 
\eqref{eq:cond_U} $\Rightarrow$  \eqref{eq:cond_U0} are obvious, we may
take $C_0=C_1 :=C_U$. Let us now assume that \eqref{eq:cond_U1'} and 
\eqref{eq:cond_U0} hold true. Fix $0<r\leq 1$.
We consider $n=0,1,2,\ldots$ such that $2^{n+1}r \leq 1$ (the set of such $n$'s
is
empty if $r> \frac{1}{2}$).
We have by \eqref{eq:cond_U1'}
\begin{align*}
\int_{2^nr \leq |z| < 2^{n+1}r} \nu(dz)
  &\leq
2^{-2n}r^{-2} \int_{2^nr \leq |z| < 2^{n+1}r} |z|^2 \nu(dz) \\
&\leq 2^{-2n}r^{-2} C_1\;  2^{(n+1)(2-\alpha)} r^{2-\alpha}
 =2^{-n\alpha} 2^{2-\alpha} C_1   r^{-\alpha} .
\end{align*}
After summing over all such $n$ we obtain
\[
\int_{r \leq |z| < 1/2} \nu(dz) \leq  \frac{2^{2-\alpha} C_1}{1-2^{-\alpha}}
r^{-\alpha}.
\]
Finally
\[
\int_{1/2 \leq |z|} \nu(dz) \leq 
4  \int_{\R^d} (|z|^2 \wedge 1) \nu(dz)
\leq 4C_0
\leq  4C_0  r^{-\alpha}.
\]
Combining the two inequalities above and \eqref{eq:cond_U1'} we get 
\eqref{eq:cond_U} with
$C_U=(\frac{2^{2-\alpha} }{1-2^{-\alpha}}+1)C_1 + 4C_0$.
\end{proof}

The following definition interpolates between measures $\nu$ which are related 
to different values of $\alpha \in (0,2)$. Such a construction is important for 
us because we want to prove comparability results which are robust in the 
sense that constants stay bounded when $\alpha\to 2^-$.

\begin{definition}\label{def:nualpha}
Assume $\nu^{\alpha_0}$ is a measure on $\cB(\R^d)$ satisfying 
\eqref{eq:cond_U} or \eqref{eq:cond_scaling} for some $\alpha_0 \in (0,2)$. For 
$\alpha_0 \leq \alpha < 2$  we define a new measure 
$\nu^{\alpha,\alpha_0}$ by
\begin{align}\label{eq:nualpha}
 \nu^{\alpha,\alpha_0} = \frac{2-\alpha}{2-\alpha_0} 
|x|^{\alpha_0 - \alpha}
\nu^{\alpha_0}(dx) \quad \text{ if } \alpha > \alpha_0 \text{ and by } \quad 
\nu^{\alpha_0,\alpha_0} = 
\nu^{\alpha_0} \,.
\end{align}
To shorten notation we write $\nu^{\alpha}$ instead of 
$\nu^{\alpha,\alpha_0}$ whenever there is no ambiguity.
\end{definition}

The above definition is consistent in the following ways. On 
the one hand, the first part of \eqref{eq:nualpha} holds true for $\alpha = 
\alpha_0$. On the other hand, for $0 < \alpha_0 < \alpha < \beta < 2$, the 
following is true: $\nu^{\beta,\alpha_0} = (\nu^{\alpha,\alpha_0})^{\beta, 
\alpha}$. This requires that $\nu^{\alpha,\alpha_0}$ itself satisfies 
\eqref{eq:cond_U} or \eqref{eq:cond_scaling} which is established in the 
following lemma.

\begin{lemma}{\ }\label{lem:nualpha_robust} Assume $\nu^{\alpha_0}$ satisfies 
\eqref{eq:cond_U}  with some $\alpha_0 \in 
(0,2)$, $C_U > 0$ or condition \eqref{eq:cond_scaling}  with some $\alpha_0 \in 
(0,2)$, $a > 1$. Assume $\alpha_0 \leq \alpha < 2$ and 
$\nu^{\alpha}$ as in \autoref{def:nualpha}.
\begin{enumerate}[label={({\alph*})}]
\item \label{enum:nualpha_partA}
If $\nu^{\alpha_0}$ satisfies \eqref{eq:cond_U}, then for every $0<b<1$, 
$0<r\leq 1$
\begin{align}\label{e.Urobust}
\int_{br\leq |z| < r} |z|^2\,\nu^\alpha(dz) &\leq \frac{2-\alpha}{2-\alpha_0} 
C_U
b^{\alpha_0-\alpha} r^{2-\alpha} \,, \\
\int_{B_r^c} \nu^\alpha(dz) &\leq \frac{2-\alpha}{2-\alpha_0} C_U r^{-\alpha} 
\,. \label{e.Urobust2}
\end{align}
\item \label{enum:nualpha_partB}
If $\nu^{\alpha_0}$ satisfies \eqref{eq:cond_U}, then $\nu^{\alpha}$ satisfies
\eqref{eq:cond_U}  with exponent $\alpha$
and  constant $13C_U (2-\alpha_0)^{-1}$. In particular, the constant does not 
depend on $\alpha$. 
\item \label{l.partC}
If $\nu^{\alpha_0}$ satisfies \eqref{eq:cond_scaling}, then  
$\nu^{\alpha}$ satisfies
\eqref{eq:cond_scaling} with exponent $\alpha$.
\end{enumerate}
\end{lemma}

\begin{proof}
Let $0<r\leq 1$ and $0<b<1$.
To prove \ref{enum:nualpha_partA}, we derive,
\begin{align*}
\int_{br\leq |z| < r} |z|^2\,\nu^\alpha(dz)
 &= \frac{2-\alpha}{2-\alpha_0} \int_{br\leq |z| < r} |z|^{2+\alpha_0-\alpha}
\,\nu^{\alpha_0}(dz) \leq
 \frac{2-\alpha}{2-\alpha_0} (br)^{\alpha_0-\alpha} \int_{B_r} |z|^{2}
\,\nu^{\alpha_0}(dz) \\
&\leq 
 \frac{2-\alpha}{2-\alpha_0} b^{\alpha_0-\alpha} C_U r^{2-\alpha},
\end{align*}
which proves \eqref{e.Urobust}. Furthermore,
\begin{align*}
\int_{B_r^c} \,\nu^\alpha(dz)
 &=
 \frac{2-\alpha}{2-\alpha_0} \int_{B_r^c}  |z|^{\alpha_0-\alpha}
\,\nu^{\alpha_0}(dz) 
\leq
 \frac{2-\alpha}{2-\alpha_0} r^{\alpha_0-\alpha} C_U r^{-\alpha_0}
\end{align*}
and \eqref{e.Urobust2} follows.
To prove part \ref{enum:nualpha_partB}, we use \eqref{e.Urobust} and conclude
\begin{align*}
\int_{B_r} |z|^2\,\nu^\alpha(dz)
 &=
 \sum_{n=0}^\infty  \int_{\frac{r}{2^{n+1}} \leq |z| < \frac{r}{2^{n}}} |z|^2
\,\nu^{\alpha_0}(dz) \leq
\frac{2-\alpha}{2-\alpha_0}C_U 2^{\alpha-\alpha_0}  r^{2-\alpha}
\sum_{n=0}^\infty 2^{n(\alpha-2)}\\
&\ =
\frac{C_U2^{\alpha-\alpha_0} r^{2-\alpha}}{2-\alpha_0}
\frac{2-\alpha}{1-2^{\alpha-2}} \leq \frac{32C_U}{3(2-\alpha_0)} r^{2-\alpha},
\end{align*}
since the function $x \mapsto \frac{x}{1-2^{-x}}$ is increasing. Furthermore, by
\eqref{e.Urobust2},
\[
\int_{B_r^c} r^2 \,\nu^\alpha(dz) \leq  \frac{2C_U}{2-\alpha_0}   r^{2-\alpha},
\]
therefore \ref{enum:nualpha_partB} follows.
Finally, part \ref{l.partC} is obvious.
\end{proof}

\begin{lemma}\label{lem:heart-upper}
Assume $\nu_j^{\alpha_0}$ for $j \in \{1,2\}$ satisfies \eqref{eq:cond_U}  
with some $\alpha_0 \in (0,2)$, $C_U > 0$. Assume $\alpha_0 \leq \alpha < 2$ 
and 
$\nu_j^{\alpha}$ as in \autoref{def:nualpha}. Then the measure  $\nu_1^\alpha 
\heartsuit \nu_2^\alpha$ satisfies \eqref{eq:cond_U} with the same exponent 
$\alpha$ and a constant depending only on $\alpha_0$,
 $C_U$, $\lambda$ and $\eta$.
\end{lemma}
\begin{proof}
By \autoref{lem:U1p}, it suffices to show that $\nu_1^\alpha \heartsuit
\nu_2^\alpha$
satisfies 
\eqref{eq:cond_U0} and \eqref{eq:cond_U1'}.
For $0<r\leq 1$ we derive
\begin{align*}
\int_{B_r} |x|^2 \nu_1^\alpha \heartsuit &\nu_2^\alpha(dx)
\leq
\frac{1}{2-\alpha}
\!\!\!\!\!\!\!\!\!\!
\iint\limits_{\lambda|y+z| \leq |y|,\,|z|\leq \eta|y+z| } |\eta(y+z)|^2
\1_{B_r}(\eta(y+z))
  |y+z|^{\alpha} 
 \, \nu_1^\alpha(dy)\, \nu_2^\alpha(dz)\\
&\leq 
\frac{1}{2-\alpha}
\iint\limits_{\lambda|y+z| \leq |y|,\,|z| \leq \eta|y+z| < r } 
   \frac{\eta^2 |y|^2}{\lambda^2}  \frac{|z|^{\alpha}}{\lambda^{\alpha}} \,
\nu_1^\alpha(dy)\, \nu_2^\alpha(dz) \\
&\leq
\frac{1}{2-\alpha}
 \frac{\eta^2}{\lambda^{2+\alpha}} \int_{B_r} |z|^{\alpha}
\int_{\frac{\lambda|z|}{\eta} \leq |y|\leq \frac{\eta|z|}{\lambda}} |y|^2 \,
\nu_1^\alpha(dy)\, \nu_2^\alpha(dz) 
\leq \frac{\eta^{4} (C_U)^2}{\lambda^4} \frac{13}{(2-\alpha_0)^2} r^{2-\alpha},
\end{align*}
where in the last passage we used parts \ref{enum:nualpha_partB} and 
\ref{enum:nualpha_partA} of
\autoref{lem:nualpha_robust}.
Furthermore, by  \eqref{e.Urobust2},
\begin{align}
\int_{\R^d \setminus B_1} \nu_1^\alpha \heartsuit \nu_2^\alpha(dx)
&\leq \frac{1}{2-\alpha}
\iint\limits_{\lambda|y+z| \leq |y|,\,|z| < \eta|y+z|} \1_{B_2\setminus
B_1}(\eta(y+z))
|y+z|^{\alpha} \, \nu_1^\alpha(y)\, \nu_2^\alpha(dz) \\
&\leq \frac{ 2^{\alpha}}{2-\alpha}
 \iint\limits_{\frac{\lambda}{\eta} \leq |y|,\,|z|}   \, \nu_1^\alpha(y)\,
\nu_2^\alpha(dz)
\leq \frac{8(C_U)^2 \eta^4}{\lambda^4(2-\alpha_0)^2}.  \qedhere
\end{align}
\end{proof}

The following lemma shows that the quadratic form w.r.t. to $\nu_1 \heartsuit 
\nu_2$ is dominated by the sum of the quadratic forms w.r.t. $\nu_1$ and 
$\nu_2$. Some enlargement of the domain is needed which is taken care of in 
\autoref{lem:Whitney} by a~covering argument. 

\begin{lemma}\label{lem:heart-form-comp}
Assume $\nu_j^{\alpha_0}$ for $j \in \{1,2\}$ satisfies \eqref{eq:cond_U} and
\eqref{eq:cond_scaling} with some $\alpha_0 \in (0,2)$, $a>1$, and $C_U > 0$. 
Assume $\alpha_0 \leq \alpha < 2$ and  $\nu_j^{\alpha}$ as in 
\autoref{def:nualpha}. Let $\eta=a^k>1$ for some $k\in \Z$. For $B=B_r(x_0)$ 
let us denote $B^*=B_{3\eta r}(x_0)$. Then with $c=4C_U\eta^6\lambda^{-4}$ it holds,
\begin{equation}
 \cE_{B}^{\nu_1 \heartsuit \nu_2}(u,u) \leq c (\cE_{B^*}^{\nu_1}(u,u) +
\cE_{B^*}^{\nu_2}(u,u))
\end{equation}
for any measurable function $u$ on $B_1$ and 
any $B$ such that $B^* \subset
B_1$.
\end{lemma}
\begin{proof}
Let $B=B_r(x_0)$ be such that $B^* \subset B_1$. In particular, this means that 
$r\leq 1/(3\eta)$.
By definition, we obtain
\begin{align}
\cE_{B}^{\nu_1\heartsuit \nu_2}(u,u) 
  &= \iint (u(x)-u(x+z))^2 \1_{B}(x) \1_{B}(x+z) \nu_1\heartsuit \nu_2(dz)\,dx
\nonumber\\
  &\leq \iiint (u(x)-u(x+\eta(y+z)))^2 \1_{B}(x) \1_{B}(x+\eta(y+z))
g_\lambda^\eta(y,z)
\nu_1(dy)\, \nu_2(dz)\,dx \nonumber\\
 &\leq  2 \iiint \Big[ (u(x)-u(x+\eta y))^2   + (u(x+\eta y)-u(x+\eta(y+z)))^2
\Big] \nonumber\\
 &\qquad\qquad\qquad\qquad\qquad\qquad \times  \1_{B}(x) \1_{B}(x+\eta(y+z))
g_\lambda^\eta(y,z) \nu_1(dy)\, \nu_2(dz)\,dx\nonumber\\
 &= 2[I_1+I_2]. \label{etmp2}
\end{align}
We may assume that $\lambda|y+z| \leq |z| < \eta |y+z| \leq 2r$ and
$\lambda|y+z| \leq |y| < \eta |y+z| \leq 2r$,
as otherwise the expression $ \1_{B}(x) \1_{B}(x+\eta(y+z)) g_\lambda^\eta(y,z)$
would
be zero. Since $2r\leq 1$, it follows that $\frac{\lambda|y|}{\eta}  < |z| \leq
\frac{\eta|y|}{\lambda} \wedge 1$.
Therefore, by changing the order of integration,
\begin{align*}
I_1 &\leq \int_{B} \int_{B_{2r}} \int_{\frac{\lambda|y|}{\eta} \vee 
\lambda|y+z| \leq |z| \leq
\frac{\eta|y|}{\lambda} \wedge 1}   (u(x)-u(x+\eta
y))^2 |y+z|^{\alpha} \nu_2(dz)\, \nu_1(dy)\,dx.
\end{align*}
We estimate the inner integral above,
\begin{align*}
J &:= \int_{\frac{\lambda|y|}{\eta} \vee \lambda|y+z| \leq |z| \leq
\frac{\eta|y|}{\lambda} \wedge 1} |y+z|^{\alpha} \, \nu_2(dz) 
 \leq 
\int_{|z| \leq \frac{\eta|y|}{\lambda} \wedge 1}
\frac{|z|^{\alpha}}{\lambda^{\alpha}} \frac{|z|^{2-\alpha}}{
\left(\frac{\lambda|y|}{\eta} \right)^{2-\alpha} }  \, \nu_2(dz) 
 \leq 
\frac{\eta^4 C_U}{\lambda^4}.
\end{align*}
Coming back to $I_1$ we obtain,
\begin{align*}
I_1 &\leq \frac{\eta^4 C_U}{\lambda^4} \int_{B}  \int_{B_{2r}} 
(u(x)-u(x+\eta y))^2 
\nu_1(dy)\, dx \\
 &= \frac{\eta^4 C_U}{\lambda^4} \eta^{\alpha} \int_{B}  \int_{B_{2\eta r}} 
(u(x)-u(x+ y))^2  \nu_1(dy)\,
dx
\leq \frac{\eta^6 C_U}{\lambda^4} \cE_{B^*}^{\nu_1}(u,u),
\end{align*}
where we used \eqref{eq:cond_scaling}
and the fact that $B_{2\eta r}\subset B_1$.

Finally, in order to estimate $I_2$, we first change variables $x=w-\eta y$,
\begin{align*}
I_2 &\leq  \int_{B } \int_{B_{2r}} \int_{B_{2r}}  (u(x+\eta
y)-u(x+\eta(y+z)))^2   \1_{B}(x+\eta(y+z)) g_\lambda^\eta(y,z) \nu_1(dy)\,
\nu_2(dz)\,dx\\
&\leq  \int_{B^*} \int_{B_{2r }}  (u(w)-u(w + \eta z))^2   \1_{B}(w+\eta z)
\int_{B_{2r}}  g_\lambda^\eta(y,z) \nu_1(dy)\, \nu_2(dz)\,dw\\
&\leq  \int_{B^*} \int_{B_{2r }}  (u(w)-u(w + \eta z))^2   \1_{B}(w+\eta z)
\int_{ \frac{\lambda |z|}{\eta} \vee \lambda|y+z| \leq |y| \leq
\frac{\eta|z|}{\lambda} \wedge 1}   |y+z|^\alpha \nu_1(dy)\, \nu_2(dz)\,dw.
\end{align*}
By symmetry, the following integral may be estimated exactly like $J$ before,
\[ 
\int_{ \frac{\lambda |z|}{\eta} \vee \lambda|y+z| \leq |y| \leq
\frac{\eta|z|}{\lambda} \wedge 1}   |y+z|^\alpha  \nu_1(dy) \leq \frac{\eta^4
C_U}{\lambda^4}.
\]
This leads to an estimate 
\begin{align*}
 I_2 &\leq   \frac{\eta^4 C_U}{\lambda^4} \int_{B^*} \int_{B_{2r}}  (u(w)-u(w +
\eta z))^2  
\1_{B}(w+\eta z) \, \nu_2(dz)\,dw \\
 &=   \frac{\eta^4 C_U}{\lambda^4} \eta^\alpha \int_{B^*} \int_{B_{2\eta r}} 
(u(w)-u(w + t))^2  
\1_{B}(w+t) \, \nu_2(dt)\,dw  \leq  \frac{\eta^6 C_U}{\lambda^4}
\cE_{B^*}^{\nu_2}(u,u),
\end{align*}
where we used \eqref{eq:cond_scaling} and the fact that $B_{2\eta r} \subset
B_1$. The result follows from \eqref{etmp2} and the obtained estimates of $I_1$
and $I_2$.
\end{proof}

\begin{lemma}\label{lem:Whitney}
Let $0<\alpha_0<\alpha<2$, $r_0>0$, $\kappa\in(0,1)$, and $\nu$ be a~measure on $\cB(\R^d)$. For $B=B_r(x)$, $x \in 
\R^d$, $r>0$, we set $B^* =
B_{\frac{r}{\kappa}}(x)$. Suppose that for some $c_\nu > 0$
\[
 \cE^\nu_{B^*}(u,u) \geq c_\nu \cE^{\mu_\alpha}_{B}(u,u),
\]
for every $0<r\leq r_0$, every $u \in L^2(B_{r_0})$, and for
every ball $B \subset B_{r_0}$ of radius
$\kappa r$.
Then there exists a constant $c=c(d,\alpha_0,\kappa)$,
such that for every ball $B \subset B_{r_0}$ of radius $r\leq r_0$ and every $u 
\in L^2(B_{r_0})$
\[
 \cE^\nu_{B}(u,u) \geq c c_\nu \cE^{\mu_\alpha}_{B}(u,u).
\]
\end{lemma}

\begin{proof}
Fix some  $0<r \leq r_0$ and a ball $D$ of radius $r$.
We take $\mathcal{B}$ to be a family of balls with the following properties.
\begin{itemize}
\item[(i)]
For some $c=c(d)$ and any $x,y\in D$, if $|x-y|<c\dist(x,D^c)$, then there
exists $B\in\mathcal{B}$
such that $x,y\in B$.
\item[(ii)]
For every $B\in\mathcal{B}$, $B^*\subset D$.
\item[(iii)]
Family $\{B^*\}_{B\in\mathcal{B}}$ has the finite overlapping property, that is,
each point of $D$ belongs to at most $M=M(d)$ balls $B^*$, where
$B\in\mathcal{B}$.
\end{itemize}
Such a family $\mathcal{B}$ may be constructed by considering Whitney
decomposition of $D$
into cubes and then covering each Whitney cube by an appropriate
family of balls.

We have
\begin{align}
\cE^\nu_{D}(u,u) &\geq
 \frac{1}{M^2} \sum_{B\in\mathcal{B}} \int_{B^*} \! \int_{B^*}
(u(x)-u(x+y))^2 \,\nu(\d y)\,\d x\nonumber \\
&\geq 
\frac{c_\nu}{M^2} (2{-}\alpha) \sum_{B\in\mathcal{B}} \int_{B} \! \int_{B}
(u(x)-u(y))^2|x-y|^{-d-\alpha}\,\d y\,\d x \nonumber\\
&\geq 
\frac{c_\nu}{M^2}  (2{-}\alpha) \int_{D} \! \int_{|x-y|<c\dist(x,D^c)}
(u(x)-u(y))^2|x-y|^{-d-\alpha}\,\d y\,\d x. \label{new}
\end{align}
By  \cite[Proposition~5 and proof of Theorem~1]{Dyda2}, we may estimate
\begin{align}
\int_{D} \! \int_{|x-y|<c\dist(x,D^c)} & (u(x)-u(y))^2|x-y|^{-d-\alpha}\,\d
y\,\d x
\nonumber\\
&\geq
c(\alpha,d) \int_{D} \! \int_{D} (u(x)-u(y))^2|x-y|^{-d-\alpha}\,\d y\,\d x
\label{old}
\end{align}
with some constant $c(\alpha,d)$. We note that in \cite[proof of
Theorem~1]{Dyda2}
the constant depends on the domain in question, but in our case, by scaling, we
can take the same
constant independent of the choice of the ball $D$. One may also check that
$c(\alpha,d)$ stays bounded when $\alpha\in [\alpha_0,2)$.
By \eqref{new} and \eqref{old} the lemma follows.
\end{proof}

For a linear subspace $E \subset \R^d$, we denote by $H_E$ the $(\dim
E)$-dimensional Hausdorff measure
on $\R^d$ with the support restricted to $E$. In particular, $H_{\{0\}} =
\delta_{\{0\}}$, the Dirac delta measure at~$0$.

\begin{proposition}\label{prop:Hausdorff}
Let $E_1, E_2 \subset\R^d$ be two linear subspaces with $E_1, E_2 \neq \{0\}$.
Assume that $\nu_j$, $j \in \{1,2\}$, are measures on $\cB(\R^d)$ of the form 
$\nu_j = f_j H_{E_j}$ satisfying $\nu_j(B_1)>0$, \eqref{eq:cond_U}, and 
\eqref{eq:cond_scaling} with $\alpha_0 \in (0,2)$, $C_U > 0$ and $a > 1$. Then 
the following is true: 
\begin{enumerate}
\item $\nu_1 \heartsuit \nu_2$ is absolutely continuous 
with respect to $H_{E_1+E_2}$ and satisfies \eqref{eq:cond_U} and 
\eqref{eq:cond_scaling}. 
\item If $\eta \geq \frac{a^2}{a-1}$ and $\lambda \leq \frac{1}{a^3+1}$, then 
$\nu_1 \heartsuit
\nu_2(B_1)>0$.
\item If $\nu_j^{\alpha_0}=\nu_j$ and $\nu_j^\alpha$ is defined as in 
\autoref{def:nualpha} for $\alpha_0 \leq \alpha < 2$, then
\begin{align}\label{e.heart-alpha-commute}
\nu_1^\alpha \heartsuit \nu_2^\alpha \geq \eta^{-2} (\nu_1^{\alpha_0} \heartsuit
\nu_2^{\alpha_0})^\alpha.
\end{align}
\end{enumerate}
\end{proposition}

\begin{proof}
Properties \eqref{eq:cond_U} and \eqref{eq:cond_scaling} follow 
from \autoref{lem:heart-upper} and \autoref{lem:heart-scaling}, respectively.
Let $E=E_1\cap E_2$ and let $F_j$ be linear subspaces such that $E_j = E \oplus
F_j$, where $j=1,2$.
For $y\in E_1$ let us write $y=Y+\tilde{y}$, where $Y\in E$ and $\tilde{y}\in
F_1$; similarly,
for $z\in E_2$ we write $z=Z+\hat{z}$, where $Z\in E$ and $\hat{z}\in F_2$.
Then for $A\subset B_2$
\begin{align}
\nu_1 &\heartsuit \nu_2(A)
 = \iiiint \1_A(\eta(Y+\tilde{y} + Z+\hat{z})) g_\lambda^\eta(Y+\tilde{y},
Z+\hat{z}) \nonumber\\
 &\qquad\qquad\qquad\qquad \times f_1(Y+\tilde{y}) f_2(Z+\hat{z}) \, H_E(dY) \,
H_E(dZ) \, H_{F_1}(d\tilde{y}) \, H_{F_2}(d\hat{z}) \nonumber\\
&= \iiint \1_A(\eta(W +\tilde{y} +\hat{z})) \left( \int
g_\lambda^\eta(Y+\tilde{y},
W-Y+\hat{z}) f_1(Y+\tilde{y}) f_2(W-Y+\hat{z})  H_E(dY) \right)  \nonumber\\
 &
\qquad\qquad\qquad\qquad\qquad\qquad\qquad\qquad\qquad
\qquad H_E(dW) \, H_{F_1}(d\tilde{y}) \, H_{F_2}(d\hat{z})
\label{e.heart-density}
\end{align}
and since $\nu_1 \heartsuit \nu_2(\R^d\setminus B_2) = 0$, the desired absolute 
continuity follows.

To show non-degeneracy, let $G_n := B_{a^{-n}} \setminus B_{a^{-n-1}}$. By 
scaling
property \eqref{eq:cond_scaling} it follows that $\nu_j(G_{n+1}) = a^\alpha 
\nu_j(G_n)$,
therefore $\nu_j(G_n) > 0$ for each $n=0,1,\ldots$.
Hence
\begin{align*}
\nu_1 \heartsuit \nu_2(B_1) 
 &\geq \frac{1}{2-\alpha_0} \int_{G_n } \int_{G_{n+2} }  \!\! 
\mathbbm{1}_{B_1}(\eta (y+z)) 
 \mathbbm{1}_{A_{|y+z|}}(y) \mathbbm{1}_{A_{|y+z|}}(z) 
  |y+z|^\alpha  \,\nu_1(dy)\,\nu_2(dz).
\end{align*}
For $(y,z) \in G_{n+2} \times G_{n}$ it holds that 
$\frac{a-1}{a^2} (|y| \vee |z|) \leq |y+z| \leq (a^3+1)(|y|\wedge |z|)$
and also $\eta(y+z) \in B_1$, provided $n$ is large enough. Therefore $\nu_1
\heartsuit \nu_2(B_1) >0$, if $\eta \geq \frac{a^2}{a-1}$ and $\lambda \leq
\frac{1}{a^3+1}$.

To prove the last part of the lemma, we calculate first the most inner
integral in \eqref{e.heart-density} corresponding to $\nu_1^\alpha\heartsuit 
\nu_2^\alpha$,
it equals
\begin{align*}
L
&:= \int g_\lambda^\eta(Y+\tilde{y}, W-Y+\hat{z}) f_1^\alpha(Y+\tilde{y})
f_2^\alpha(W-Y+\hat{z})  H_E(dY)  \\
& = \frac{2-\alpha}{(2-\alpha_0)^2} 
 \int |W +\tilde{y} +\hat{z}|^{\alpha} |Y+\tilde{y}|^{\alpha_0-\alpha}
|W-Y+\hat{z}|^{\alpha_0-\alpha} \1(\ldots) \\
&\qquad\qquad\qquad\qquad \times f_1^{\alpha_0}(Y+\tilde{y})
f_2^{\alpha_0}(W-Y+\hat{z})  H_E(dY),
\end{align*}
where we used an abbreviation
\[
 \1(\ldots) := \1_{A_{|W +\tilde{y} +\hat{z}|}}(Y+\tilde{y}) \1_{A_{|W
+\tilde{y} +\hat{z}|}}(W-Y+\hat{z}).
\]
On the other hand, the 
most inner
integral in \eqref{e.heart-density} corresponding to 
$(\nu_1^{\alpha_0}\heartsuit \nu_2^{\alpha_0})^\alpha$
is
\begin{align*}
R &:= \frac{2-\alpha}{2-\alpha_0} (\eta|W +\tilde{y} 
+\hat{z}|)^{\alpha_0-\alpha}
\\
&\qquad\qquad\qquad\times \int g_\lambda^\eta(Y+\tilde{y}, W-Y+\hat{z})
f_1^{\alpha_0}(Y+\tilde{y}) f_2^{\alpha_0}(W-Y+\hat{z})  H_E(dY)  \\
& = \frac{(2-\alpha)  \eta^{\alpha_0-\alpha}}{(2-\alpha_0)^2}
 \int |W +\tilde{y} +\hat{z}|^{2\alpha_0 - \alpha} \1(\ldots) 
f_1^{\alpha_0}(Y+\tilde{y}) f_2^{\alpha_0}(W-Y+\hat{z})  H_E(dY).
\end{align*}
Inequality \eqref{e.heart-alpha-commute} follows now from the following
estimate,
\[
 |Y+\tilde{y}|^{\alpha_0-\alpha} |W-Y+\hat{z}|^{\alpha_0-\alpha}  \1(\ldots)
\geq (\eta |W +\tilde{y} +\hat{z}|)^{2(\alpha_0 - \alpha)}  \1(\ldots)
\]
and the fact that both sides of \eqref{e.heart-alpha-commute} are zero on 
$\R^d\setminus B_2$.
\end{proof}

\begin{proof}[Proof of \autoref{theo:U1L1implyA} -- lower bound in
\eqref{eq:assum_comp}]
We recall from \autoref{subsec:intro_comp} that we may and do assume that $f_k$
are symmetric, i.e., $f_k(x)=f_k(-x)$ for all $x$. By \autoref{prop:Hausdorff} 
it follows 
that the measure
\[
\nu:=(f_1 H_{E_1}) \heartsuit (f_2 H_{E_2}) \heartsuit \ldots \heartsuit (f_n
H_{E_n})
\]
satisfies \eqref{eq:cond_U} and \eqref{eq:cond_scaling} and has a~density
$h$ with respect to the Lebesgue measure on $\cB(\R^d)$ with $\int_{B_1} 
h(x)\,dx > 
0$, if $\eta$ is large enough and $\lambda$ small enough. We will show that the 
measure $\nu \heartsuit \nu$ possesses a density $h^\heartsuit$ 
with $h^\heartsuit(x) \geq c |x|^{-d-\alpha_0}$ for all $x \in B_1 \setminus 
\{0\}$ and some positive constant $c$ to be specified. This, together with the 
preliminary results, will establish the assertion.

Condition \eqref{eq:cond_scaling} for $\nu$ implies that 
$h(ax)=a^{-d-\alpha_0}
h(x)$ if $x\in B_{1/a}$.
Therefore $\int_{G_0} h(x)\,dx > 0$, where $G_0=B_1\setminus B_{1/a}$.
Define $h^{G_0}(x)=h(x)\1_{G_0}(x) \wedge 1$. The function
\[
 x \mapsto h^{G_0}\ast h^{G_0}(x) = \int h^{G_0}(y-x) h^{G_0}(y)\,dy
\]
is continuous and strictly positive at $0$. Thus there exists $\delta \in
(0,(2a)^{-1})$ and $\varepsilon >0$
such that
\[
 h^{G_0}\ast h^{G_0}(x) \geq \varepsilon \qquad \text{for $x\in B_\delta$.}
\]
We consider the measure $\nu \heartsuit \nu$, it has a~density $h^\heartsuit$
with respect to the Lebesgue measure on $\cB(B_2)$
given by formula, cf. \eqref{e.L-density},
\begin{align*}
 h^\heartsuit(x) &= \eta^{-2d} \int g_\lambda^\eta(\tfrac{w}{\eta},
\tfrac{x-w}{\eta})
h(\tfrac{w}{\eta}) h(\tfrac{x-w}{\eta}) \,dw \\
 &\geq \eta^{2\alpha_0} \int_{G_0} g_\lambda^\eta(\tfrac{w}{\eta},
\tfrac{x-w}{\eta})
\1_{G_0}(x-w)  h(w) h(x-w)\,dw \\
&=\frac{\eta^{\alpha_0}}{2-\alpha_0} \int_{G_0} 
 |x|^{\alpha_0}
 \mathbbm{1}_{A_{|x|}}(w)  \mathbbm{1}_{A_{|x|}}(x-w)
\1_{G_0}(x-w) h(w) h(x-w)\,dw.
\end{align*}
Suppose $\eta \geq a^2/\delta$ and $\lambda\leq 1/(a\delta)$. Then for $x\in B_\delta \setminus
B_{\delta/a^2}$ and $w\in {G_0}$ such that $x-w\in {G_0}$ it holds
\[
 \mathbbm{1}_{A_{|x|}}(w)  \mathbbm{1}_{A_{|x|}}(x-w) = 1.
 \]

This leads to the following estimate
\begin{align*}
 h^\heartsuit(x) 
 &\geq \frac{\eta^{\alpha_0} \delta^{\alpha_0} a^{-2\alpha_0} }{2-\alpha_0} 
h^{G_0} \ast h^{G_0}(x)
 \geq
\frac{\varepsilon}{2-\alpha_0} , \qquad \text{for $x\in B_\delta \setminus
B_{\delta/a^2}$.}
\end{align*}
For $x\in B_1\setminus \{0\}$ let $k\in \Z$ be such that $\frac{\delta}{a^2}<
|x|a^k < \delta < |x|a^{k+1}$. Then, by scaling \eqref{eq:cond_scaling},
\[
 h^\heartsuit(x) = a^{k(d+\alpha_0)} h^\heartsuit(xa^k) 
\geq 
\frac{a^{k(d+\alpha_0)} \varepsilon}{2-\alpha_0}
 \geq \frac{\delta^{d+\alpha_0}\varepsilon}{a^{2d+2\alpha_0}(2-\alpha_0)}
|x|^{-d-\alpha_0}.
\]
Now from \autoref{lem:heart-form-comp} and \autoref{lem:Whitney} it follows that
for any $B\subset B_1$
\begin{equation}\label{e.end4}
 \cE_B^{\mu_{\alpha_0}}(u,u) \leq c \cE_B^{\nu_*}(u,u),
\end{equation}
with $c=c((f_j), (E_j))$.

Finally, to obtain a robust result, we observe that by 
\eqref{e.heart-alpha-commute}
\begin{align*}
\underbrace{ (\nu_*)^\alpha \heartsuit \ldots \heartsuit (\nu_*)^\alpha }_\text{$2n$ 'factors'}
&\geq
\eta^{-2(2n-1)} (\,\underbrace{\nu_* \heartsuit \ldots \heartsuit \nu_*}_\text{$2n$ 'factors'}\, )^\alpha \\
&\geq 
 \eta^{-2(2n-1)} \frac{2-\alpha}{2-\alpha_0} |x|^{\alpha_0-\alpha}
   \frac{\delta^{d+\alpha_0}\varepsilon}{a^{2d+2\alpha_0}} |x|^{-d-\alpha_0} 
\1_{B_1}(x)\,dx.
\end{align*}
This together with  \autoref{lem:heart-form-comp} and \autoref{lem:Whitney} 
gives us
\[
 \cE_B^{\alpha}(u,u) \leq c \cE_B^{(\nu_*)^\alpha}(u,u),
\]
with the constant $c$ \emph{not depending} on $\alpha\in [\alpha_0,2)$.
\end{proof}

Let us show that the assumptions of \autoref{theo:U1L1implyA} are not necessary
for \eqref{eq:assum_comp} and \eqref{eq:assum_cutoff} to hold. This is true 
because the condition \eqref{eq:assum_comp} relates to integrated 
quantities but does not require pointwise bounds on the density of $\mu(x,\d 
y)$.


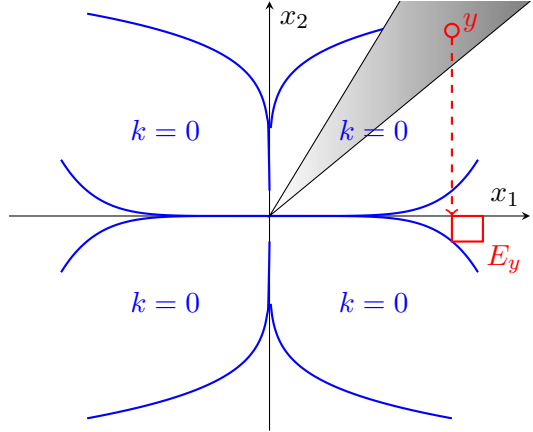
\begin{figure}

\begin{center}

\begin{tikzpicture}
\begin{axis}[
smooth,
axis lines=middle,
xmin=-1,
xmax=1,
ymin=-1,
ymax=1,
xlabel=$x_1$,
ylabel=$x_2$,
xtick=\empty,
ytick=\empty,
]
\addplot[smooth, thick, blue, domain=0:0.7, samples=150] {x^(1/6)};
\addplot[smooth, thick, blue, domain=0:0.7, samples=150] {-x^(1/6)};
\addplot[smooth, thick, blue, domain=-0.7:0, samples=150] {(-x)^(1/6)};
\addplot[smooth, thick, blue, domain=-0.7:0, samples=150] {-(-x)^(1/6)};
\addplot[smooth, thick, blue, domain=0:0.8, samples=150] {x^6};
\addplot[smooth, thick, blue, domain=0:0.8, samples=150] {-x^6};
\addplot[smooth, thick, blue, domain=-0.8:0, samples=150] {(-x)^6};
\addplot[smooth, thick, blue, domain=-0.8:0, samples=150] {-(-x)^6};


\shadedraw[right color = gray, left color = white] (axis cs: 0,0) 
-- (axis cs: 1,1)  -- (axis cs: 1,2) -- cycle;

\draw [red, thick] (axis cs: 0.7, -0.119) rectangle (axis cs: 0.819,0);
\node [red, thick] at (axis cs: 0.9,-0.2){\bf $E_y$};

\draw [thick, red, dashed, o->] (axis cs: 0.7,0.9) -- (axis cs: 0.7,0);  
\node [red, thick, right] at (axis cs: 0.7,0.9){\bf $y$};

\node [blue, thick] at (axis cs: 0.4,-0.4){\bf $k = 0$};
\node [blue, thick] at (axis cs: -0.4,-0.4){\bf $k = 0$};
\node [blue, thick] at (axis cs: 0.4,0.4){\bf $k = 0$};
\node [blue, thick] at (axis cs: -0.4,0.4){\bf $k = 0$};

\end{axis}
\end{tikzpicture}

\end{center}

\caption{Support of the kernel $k$ (with 
$b=1/6$) consisting of four thorns. The set $P$ from the proof below is 
shown, too.} \label{fig:thorn}
\end{figure}

\begin{example}\label{ex:cone}

Let $b\in(0,1)$ and
\[
 \Gamma = \{ (x_1,x_2)\in \R^2| |x_2|\geq |x_1|^b \textrm{ or } |x_1|\geq
|x_2|^b \}.
\]
We consider the following function
\begin{equation}
k(z) = (2{-}\alpha) \ind_{\Gamma \cap B_1}(z) |z|^{-2-\beta}, \quad z\in \R^2,
\end{equation}
where $\beta = \alpha-1+1/b$, see \autoref{fig:thorn}. As we will show, for 
such a function $k$ conditions (A) and (B) are satisfied. We have, for $0<r<1$
\begin{align}
\int_{B_r} |z|^2 k(z)\,dz &\leq 8(2{-}\alpha)\int_0^r \int_0^{x^{1/b}}
(x^2+y^2)^{-\beta/2} \,\d y\,\d x\nonumber\\
&\leq 8(2{-}\alpha)\int_0^r \int_0^{x^{1/b}} x^{-\beta} \,\d y\,\d x = 
 8 r^{2-\alpha}, \label{qU1}
\end{align}
hence $k$ satisfies \eqref{eq:cond_U1'} with $C_1 = 8$.
Since \eqref{eq:cond_U0} is clear, from \autoref{lem:U1p} we conclude that $k$ 
satisfies
\eqref{eq:cond_U}.

Let
\[
 P = \{x\in B_{1/4}| 0< x_1 < x_2 < 2x_1 \}
\]
and for $y=(x_1,x_2)\in P$, let
\[
 E_y =[x_1, x_1+x_1^{1/b}] \times [-x_1^{1/b}, 0].
\]
It is easy to check that if $y\in P$ and $z\in E_y$, then
\[
 \frac{|y|}{3} \leq |z| \leq 4|y| \textrm{,}\qquad 
  \frac{|y|}{3} \leq |y-z| \leq 4|y| \qquad\textrm{and}\qquad 
z,\, y-z \in \Gamma\cap B_1 \,.
\]
Let $\eta=4$ and $\lambda=\frac{1}{3}$. Then for $y\in P$
\begin{align*}
 k\heartsuit k(\eta y)
&=
\frac{|y|^\alpha}{2-\alpha} \int \1_{A_{|y|}}(z) \1_{A_{|y|}}(y-z)  
(2-\alpha)^2 
 \1_{\Gamma\cap B_1}(z) \1_{\Gamma\cap B_1}(y-z) |z|^{-2-\beta} 
|y-z|^{-2-\beta}\, dz\\
&\geq
(2-\alpha) |y|^\alpha \int_{E_y} |z|^{-2-\beta} |y-z|^{-2-\beta}\, dz \\
&\geq
(2-\alpha) |y|^\alpha (4|y|)^{2(-2-\beta)} x_1^{2/b} \\
&\geq
(2-\alpha) 3^{-2/b} 4^{-4-2\beta} |y|^{-2-\alpha} \geq 4^{-6}12^{-2/b} 
(2-\alpha)  |y|^{-2-\alpha}.
\end{align*}
\end{example}

In the following example we provide a condition that implies comparability 
of corresponding quadratic forms but which is not covered by 
\autoref{theo:U1L1implyA}.

\begin{example}
For a measure $\nu$ on $\cB(\R^d)$ with a~density $k$ with respect 
to the Lebesgue measure we formulate the following condition:
\begin{align} 
\label{eq:assum_L1} 
\begin{split}
& \text{There exist $a>1$ and $C_2, C_3 > 0$ such that every annulus 
$B_{a^{-n+1}}\setminus B_{a^{-n}}$ ($n=0,1,\ldots$)} \\
&\text{contains a ball $B_n$ with radius $C_2 a^{-n}$, such that}\\
&\qquad\qquad k(z) \geq C_3(2-\alpha) |z|^{-d-\alpha}, \quad z \in B_n.
\end{split}
\end{align}
The following proposition provides a substitute for \autoref{theo:U1L1implyA}.
\begin{proposition}
Let $a>1$, $\alpha_0 \in (0,2)$, $\alpha\in [\alpha_0,2)$, and  $C_U, C_2, C_3 
>0$.
 Let $\mu=(\mu(x,\cdot) )_{x\in \R^d}$ be a family of measures on $\R^d$ which
satisfies \eqref{eq:mu-symmetry}.
Furthermore, we assume that there exist measures $\loB$ and $\upB$ with property
\eqref{eq.lo-up-meas}, such that
\eqref{eq:cond_U} and \eqref{eq:assum_L1} hold with exponent $\alpha$
and the constants $C_U, C_2, C_3 $.
 Then there is $A = A(a,\alpha_0,C_U, C_2, C_3 ) \geq 1$ not depending on 
$\alpha$
 such that \eqref{eq:assum_comp} hold. 
\end{proposition}

\begin{proof}
We fix $\lambda < 2/C_2 \wedge 1$ and $\eta \geq 2a^2/C_2 \vee 1$.
Let for some $n\in \{0,1,\ldots\}$,
\[
 \frac{C_2}{2} a^{-n-1} \leq |y| \leq  \frac{C_2}{2} a^{-n},
\]
and assume that $\eta y \in B_2$. 
By formula \eqref{e.L-density}, we obtain
\begin{align*}
k\heartsuit k(\eta y)
 &\geq
\frac{\eta^{-d}|y|^\alpha}{2-\alpha} \int \1_{A_{|y|}}(y-z) \1_{A_{|y|}}(z) 
k(y-z) k(z)\,dz.
\end{align*}
Let us denote by $B^o_n$ the ball concentric with $B_n$, but with radius $C_2 
a^{-n}/2$ (that is, $B^o_n$ is twice smaller than $B_n$).
We observe that if $z\in B^o_n$, then $y-z\in B_n$. Furthermore, by our choice 
of $\lambda$ and $\eta$ it follows that
\[
\lambda|y| \leq |y-z| < \eta |y|, \qquad   \lambda|y| \leq |z| < \eta |y|,
\qquad \text{if $z\in B^o_n$,}
\]
that is, $y-z$, $z\in A_{|y|}$ for $z\in B^o_n$. Hence
\begin{align*}
k\heartsuit k(\eta y)
 &\geq
 \frac{\eta^{-d}|y|^\alpha}{2-\alpha} C_3^2 (2-\alpha)^2 \int_{B^0_n} 
|y-z|^{-d-\alpha}  |z|^{-d-\alpha} \,dz\\
&\geq 
\frac{C_3^2 \eta^{-d} (2-\alpha) C_2^{2d+2\alpha}}{ 2^{2d+2\alpha} 
a^{3d+4\alpha}} |y|^{-d-\alpha} \\
&\geq C(\alpha_0, d, C_2, C_3, \eta, a) (2-\alpha) |y|^{-d-\alpha},
\end{align*}
or, equivalently, for $w\in B_2$
\[
 k\heartsuit k(w) \geq C'(\alpha_0, d, C_2, C_3, \eta, a) (2-\alpha) 
|w|^{-d-\alpha}.
\]
By \autoref{lem:heart-form-comp} and \autoref{lem:Whitney} we conclude that 
the 
lower estimate in \eqref{eq:assum_comp} holds. The upper 
estimate is in turn a~consequence of \autoref{prop:upper}.
\end{proof}
\end{example}

\section{Global comparability results for nonlocal quadratic 
forms}\label{sec:appendix}

In this section we provide a global comparability result, i.e. we study
comparability in the whole $\R^d$. This result is not needed for the other
results in this article, however it contains an interesting and useful
observation. 

\begin{proposition}\label{prop:upperRd}
Assume \eqref{eq:cond_U} holds. Then there exists a constant $c=c(\alpha,d, 
C_U)$ such 
that
\begin{equation}\label{Rdupper2Rd}
 \cE^\mu(u,u) \leq c (\cE^{\mu_\alpha}(u,u) + \|u\|_{L^2(\R^d)}^2) 
\qquad \text{ for every } u\in L^2(\R^d)\,.
\end{equation}
Furthermore, if \eqref{eq:cond_U} is satisfied for \emph{all} $r>0$, then  for 
every $u\in L^2(\R^d)$
\begin{equation}\label{RdupperRd}
 \cE^\mu(u,u) \leq c \cE^{\mu_\alpha}(u,u)\,.
\end{equation}
If the constant $C_U$ in  \eqref{eq:cond_U} is independent of $\alpha \in 
(\alpha_0,2)$, where $\alpha_0>0$,
then so are the constants in \eqref{Rdupper2Rd} and \eqref{RdupperRd}.
\end{proposition}
\begin{proof}
By $E$ we denote the identity operator from $H^{\alpha/2}(\R^d)$ to itself.
One easily checks that the proof of \autoref{prop:upper} from
\eqref{Plproof} until \eqref{EkdEa} works also in the present case of $D=\R^d$.
Hence \eqref{Rdupper2Rd} follows.

To prove \eqref{RdupperRd} we observe that if \eqref{eq:cond_U} holds for all
$r>0$, 
then also
\eqref{largexi} holds for \emph{all} $\xi \neq 0$, we plug it into \eqref{Plproof2}
and we are done.
\end{proof}

We consider the following condition.
\begin{itemize}
\item[\as{2}]
 There exists $c_0 >0$ such that for all
$h \in S^{d-1}$ and all $0<r< r_0$
\begin{equation}\label{A0}
 \int_{\R^d} r^2 \sin^2\Big(\frac{h \cdot z}{r}\Big) \loB(dz) \geq c_0
r^{2-\alpha}.
\end{equation}
\end{itemize}

Clearly \eqref{eq:assum_L1}  implies \as{2} for $r_0=1$, and if $C_3$ is 
independent of $\alpha \in (\alpha_0,2)$, where $\alpha_0>0$,
then so is $c_0$.
Condition \as{2} is also satisfied if for all $h \in S^{d-1}$ and all $0<r< r_0$
\begin{equation}\label{A2}
 \int_{B_r(0)} |h \cdot z|^2 \loB(dz) \geq c_2 r^{2-\alpha}.
\end{equation}
We note that \eqref{comp0} under condition \eqref{A2} has been proved in
\cite{AbelsHusseini}
by Abels and Husseini.
The following theorem extends their result by giving a \emph{characterization}
of kernels $\loB$
admitting comparability \eqref{comp0}.
We stress that $r_0=\infty$ is allowed, and in such a case we put 
$\frac{1}{r_0^\alpha} = 0$.

\begin{theorem}\label{thm:Rd}
Let $0<r_0\leq \infty$.
If \as{2} holds, then
\begin{equation}\label{comp0}
\cE^{\mu_\alpha}(u,u) \leq \frac{1}{c_0} \cE^\mu(u,u) +
\frac{2^\alpha}{r_0^\alpha} \|u\|_{L^2}^2\,,
\quad u\in C_c^1(\R^d).
\end{equation}
Conversely, if for some $c<\infty$
\begin{equation}\label{comp0conv}
\cE^{\mu_\alpha}(u,u) \leq c 
\cE^{\loB}(u,u) +
\frac{2^\alpha}{r_0^\alpha} \|u\|_{L^2}^2\,,\quad
u\in \mathcal{S}(\R^d),
\end{equation}
then \as{2} holds.
\end{theorem}
\begin{proof}
Recalling that $(u(\cdot+z))^\wedge(\xi) = e^{i\xi\cdot z}\hat{u}(\xi)$
and using Plancherel formula
we obtain
\begin{align}
\cE^\mu(u,u) &\geq
\iint (u(x)-u(x+z))^2 \,dx\, \loB(dz) \nonumber\\
 &=\iint |e^{i\xi\cdot z}-1|^2 |\hat{u}(\xi)|^2\,d\xi
 \,  \loB(dz) \nonumber\\
 &=\int \left( \int 4\sin^2\Big(\frac{\xi\cdot z}{2}\Big) \loB(dz) \right)
|\hat{u}(\xi)|^2\,d\xi. \label{Ek}
\end{align}
If \as{2} holds, then for all $|\xi| > 2/ r_0$
\begin{align*}
 \int 4\sin^2\Big(\frac{\xi\cdot z}{2}\Big) \loB(dz) &\geq
\frac{4c_0}{2^\alpha} |\xi|^\alpha \geq c_0|\xi|^\alpha.
\end{align*}
For $|\xi| \leq 2/ r_0$ we have $|\xi|^\alpha \leq (2/r_0)^\alpha$.
Inequality \eqref{comp0}  follows from 
\begin{equation}\label{Ealpha}
\frac{\cA_{d,-\alpha}}{2\alpha(2{-}\alpha)} \cE^\alpha_{\R^d}(u,u) =
\int_{\R^d}|\xi|^\alpha |\hat{u}(\xi)|^2\,d\xi.
\end{equation}
Now we prove the converse. Assume \eqref{comp0conv}.
By \eqref{Ek}, the right hand side of \eqref{comp0conv} equals
\[
 \int \left( c\int 4\sin^2\Big(\frac{\xi\cdot z}{2}\Big) \loB(dz)  +
\frac{2^\alpha}{r_0^\alpha}\right) |\hat{u}(\xi)|^2\,d\xi,
\]
hence by \eqref{Ealpha} and \eqref{comp0conv} we obtain that
\begin{equation}\label{closetoA0}
c\int 4\sin^2\Big(\frac{\xi\cdot z}{2}\Big) \loB(dz)  +
\frac{2^\alpha}{r_0^\alpha} \geq |\xi|^\alpha, 
  \quad \textrm{for a.e. $\xi\in\R^d$.}
\end{equation}
By continuity of the function
\[
 \R^d\setminus\{0\} \ni \xi \mapsto \int 4\sin^2\Big(\frac{\xi\cdot z}{2}\Big)
\loB(dz),
\]
\eqref{closetoA0} holds for all $\xi\in\R^d$. For $|\xi|\geq
2^{1+1/\alpha}r_0^{-1}$
we have by \eqref{closetoA0}
\[
c\int 4\sin^2\Big(\frac{\xi\cdot z}{2}\Big) \loB(dz)   \geq
\frac{|\xi|^\alpha}{2},
\]
and hence \as[2^{-1/\alpha}r_0]{2}{}  holds with $c_0=2^{\alpha-3}c^{-1}$.
Since
\[
 \sin^2\Big(\frac{h \cdot z}{2r}\Big) \geq \frac{1}{4}  \sin^2\Big(\frac{h \cdot 
z}{r}\Big),
\]
also \as{2} holds with \emph{some} constant $c_0$.
\end{proof}


%
\enlargethispage*{4 \baselineskip}
\bibliographystyle{abbrv}
\bibliography{bib_file_mega}

\end{document}